\documentclass[a4paper,12pt]{amsart}

\usepackage{amssymb}
\usepackage{graphicx}
\usepackage{float}
\usepackage{amsmath}
\usepackage{array}
\usepackage{multirow}
\usepackage{mathrsfs}
\usepackage{textcomp}
\usepackage[bottom]{footmisc}
\usepackage{xcolor}
\usepackage{cite}

\newcolumntype{M}[1]{>{\raggedright}m{#1}}

\DeclareMathAlphabet{\mathpzc}{OT1}{pzc}{m}{it}

\newtheorem{theorem}{Theorem}[section]

\newtheorem{proposition}[theorem]{Proposition}
\newtheorem{corollary}[theorem]{Corollary}

\newtheorem{claim}[theorem]{Claim}
\theoremstyle{definition}
\newtheorem{definition}[theorem]{Definition}
\newtheorem{example}[theorem]{Example}

\theoremstyle{remark}
\newtheorem{remark}[theorem]{Remark}

\numberwithin{equation}{section}

\allowdisplaybreaks[4]

\begin{document}

\title{Zeta-function and $\mu^*$-Zariski pairs of surfaces}

\author{Christophe Eyral and Mutsuo Oka}

\address{C. Eyral, Institute of Mathematics, Polish Academy of Sciences, ul. \'Sniadeckich 8, 00-656 Warsaw, Poland}  
\email{cheyral@impan.pl} 
\address{M. Oka, Professor Emeritus of Tokyo Institute of Technology, 3-19-8 Nakaochiai, Shinjuku-ku, Tokyo 161-0032, Japan}   
\email{okamutsuo@gmail.com}

\thanks{}



\begin{abstract}
A \emph{Zariski pair of surfaces} is a pair of complex polynomial functions in $\mathbb{C}^3$ which is obtained from a \emph{classical Zariski pair} of projective curves $f_0(z_1,z_2,z_3)=0$ and $f_1(z_1,z_2,z_3)=0$ of degree $d$ in $\mathbb{P}^2$ by adding a same term of the form $z_i^{d+m}$ ($m\geq 1$) to both $f_0$ and $f_1$ so that the corresponding affine surfaces of $\mathbb{C}^3$ \textemdash\ defined by $g_0:=f_0+z_i^{d+m}$ and $g_1:=f_1+z_i^{d+m}$ \textemdash\ have an isolated singularity at the origin and the same zeta-function for the monodromy associated with their Milnor fibrations (so, in particular, $g_0$ and $g_1$ have the same Milnor number). In the present paper, we show that if $f_0$ and $f_1$ are ``convenient'' with respect to the coordinates $(z_1,z_2,z_3)$ and if the singularities of the curves $f_0=0$ and $f_1=0$ are Newton non-degenerate in some suitable local coordinates, then $(g_0,g_1)$ is a \emph{$\mu^*$-Zariski pair of surfaces}, that is, a Zariski pair of surfaces whose polynomials $g_0$ and $g_1$ have the same Teissier's $\mu^*$-sequence but lie in different path-connected components of the $\mu^*$-constant stratum. 
To this end, we prove a new general formula that gives, under appropriate conditions, the Milnor number of functions of the above type, and we show (in a general setting) that two polynomials functions lying in the same path-connected component of the $\mu^*$-constant stratum can always be joined by a ``piecewise complex-analytic path''.
\end{abstract}

\maketitle

\markboth{C. Eyral and M. Oka}{Zeta-function and $\mu^*$-Zariski pairs of surfaces}

%%%%%%%%%%%%%%%%%%%%%%%%%%%%%%%%%%%%%%%%%%%%%%%%%%%%%%%%%%%%%%%%%%%%%%%%%%%%%%%%%%%%%%%%%%%%%
\section{Introduction}
%%%%%%%%%%%%%%%%%%%%%%%%%%%%%%%%%%%%%%%%%%%%%%%%%%%%%%%%%%%%%%%%%%%%%%%%%%%%%%%%%%%%%%%%%%%%%

Consider two reduced homogeneous polynomial functions $f_0(z_1,z_2,z_3)$ and $f_1(z_1,z_2,z_3)$ of degree $d$ in $\mathbb{C}^3$ which are ``convenient'' (i.e., the Newton boundaries of $f_0$ and $f_1$ intersect each coordinate axis) and such that the corresponding curves $C_0$ and $C_1$ in the complex projective plane~$\mathbb{P}^2$ makes a ``Zariski pair''. This means that there are regular neighbourhoods $N(C_0)$ and $N(C_1)$ of $C_0$ and $C_1$, respectively, such that the pairs $(N(C_0),C_0)$ and $(N(C_1),C_1)$ are homeomorphic while the pairs $(\mathbb{P}^2,C_0)$ and $(\mathbb{P}^2,C_1)$ are not (see \cite{Z,Artal,Artaletal}). We suppose that the singularities of the curves are located in $z_1z_2z_3\not=0$.
Now, add a same term of the form $z_i^{d+m}$ (where $m$ is an integer $\geq 1$) to both $f_0$ and $f_1$ so that the corresponding affine surfaces of $\mathbb{C}^3$, defined by $g_0:=f_0+z_i^{d+m}$ and $g_1:=f_1+z_i^{d+m}$ respectively, have an isolated singularity at the origin.
As in \cite{O2,O3}, we say that the pair $(g_0,g_1)$ is a \emph{Zariski pair of surfaces} (or a \emph{Zariski pair of links}) if $g_0$ and $g_1$ have the same zeta-function for the monodromy associated with their Milnor fibrations (so, in particular, $g_0$ and $g_1$ have the same Milnor number). 
The main, but not unique, goal of the present paper is to show that if the singularities of the curves $C_0$ and $C_1$ are Newton non-degenerate in some suitable local coordinates, then $(g_0,g_1)$ is a Zariski pair of surfaces for which $g_0$ and $g_1$ have the same Teissier's $\mu^*$-sequence while lying in different path-connected components of the corresponding $\mu^*$-constant stratum (see Theorem \ref{mt2}). As in \cite{O3}, we call such a special Zariski pair of surfaces a \emph{$\mu^*$-Zariski pair of surfaces}.

The main tool we use in the proof is a formula, established by the second named author in \cite{O2}, which gives the zeta-function of the monodromy associated with the Milnor fibration of an ``almost Newton non-degenerate function''. (The class of almost Newton non-degenerate functions, less rigid than the class of Newton non-degenerate functions,  enjoys many interesting properties as shown in \cite{O2,O3}. The components $g_0$ and $g_1$ of our $\mu^*$-Zariski pair of surfaces are such functions.)  We shall apply this formula for the zeta-function in order to show a crucial step of the proof of Theorem \ref{mt2}. This step is another general formula (hereafter referred to as ``shift formula'', see Theorem~\ref{mt1}) that gives the Milnor number of a function in~$\mathbb{C}^n$ of the form $g=f+z_i^{d_i+m}$ ($m\geq 1$). Here, $f$ is a weighted homogeneous polynomial (for some weight $\mathbf{w}=(w_1,\ldots,w_n)$) of the form 
\begin{equation*}
f=a_1z_1^{d_1}+\cdots +a_nz_n^{d_n}+\sum_{\alpha=(\alpha_1,\ldots,\alpha_n)} a_\alpha z_1^{\alpha_1}\cdots z_n^{\alpha_n}
\end{equation*}
such that the singular locus of $V:=f^{-1}(0)$ is $1$-dimensional and for any proper subset $I\subsetneq \{1,\ldots,n\}$ the restriction of $f$ to $\mathbb{C}^I:=\{(z_1,\ldots,z_n)\in\mathbb{C}^n\mid z_i=0 \mbox{ and } i\notin I\}$ is Newton non-degenerate. 
We also assume that $f$ satisfies the following ``Newton pre-non-degeneracy condition''. Take a toric modification $\hat\pi\colon X\to\mathbb{C}^n$ compatible with the dual Newton diagram of $f$, and consider the divisor $\hat E(\mathbf{w})$ associated with the weight $\mathbf{w}$. We say that $f$ is Newton pre-non-degenerate if each singularity $\mathbf{p}\in E(\mathbf{w}):=\hat E(\mathbf{w})\cap \widetilde{V}$ (where $\widetilde{V}$ denotes the strict transform of $V$) is convenient and Newton non-degenerate in suitable local coordinates (i.e., for appropriate local coordinates near $\mathbf{p}$, the hypersurface $E(\mathbf{w})$ of $\hat E(\mathbf{w})$ is defined by a convenient Newton non-degenerate function). For more details, see Definition \ref{def-psND}. Under these assumptions on $f$, we show that the function $g=f+z_i^{d_i+m}$ is almost Newton non-degenerate, so that, in order to obtain its Milnor number, it is enough to compute the degree of its zeta-function \textemdash\ a zeta-function which can be effectively computed using \cite{O2}. 

Another ingredient that plays an important role in the proof of Theorem \ref{mt2} is the following property of $\mu^*$-constant strata.
Suppose $f$ and $f'$ are polynomial functions on $\mathbb{C}^n$ 
that vanish at the origin. If $f$ and $f'$ (as germs of analytic functions at the origin) lie in the same path-connected component of the $\mu^*$-constant stratum, then $f$ and $f'$ can always be joined by a ``piecewise complex-analytic path'' (see Definition \ref{def-pcap} and the comment after it for the precise meaning). Up to our knowledge, this property has never been observed so far. We shall prove it in Section \ref{sect-smcmscs} (see Theorem \ref{mt3}). Certainly, this latter result as well as a similar one concerning the $\mu$-constant stratum (which we also prove in Section \ref{sect-smcmscs}) may be useful in many situations in singularity theory.

\tableofcontents

%%%%%%%%%%%%%%%%%%%%%%%%%%%%%%%%%%%%%%%%%%%%%%%%%%%%%%%%%%%%%%%%%%%%%%%%%%%%%%%%%%%%%%%%%%%%%
\section{Formula for the zeta-function of an almost Newton non-degenerate function}\label{OkaFormula}
%%%%%%%%%%%%%%%%%%%%%%%%%%%%%%%%%%%%%%%%%%%%%%%%%%%%%%%%%%%%%%%%%%%%%%%%%%%%%%%%%%%%%%%%%%%%%

In this section we recall a formula, established by the second named author in \cite{O2}, which gives the zeta-function of the monodromy associated with the Milnor fibration at the origin of an almost Newton non-degenerate function. This formula generalizes the classical Varchenko formula, given in \cite{V}, which is about Newton non-degenerate functions. It will play a crucial role to establish the shift formula for the Milnor number mentioned in the introduction and to construct our $\mu^*$-Zariski pair.

Throughout this section, let $\mathbf{z}=(z_1,\ldots,z_n)$ be coordinates for $\mathbb{C}^n$ ($n\geq 2$), and let $f(\mathbf{z})=\sum_\alpha a_\alpha \mathbf{z}^\alpha$ be a non-constant analytic function defined in a neighbourhood $U$ of the origin $\mathbf{0}\in \mathbb{C}^n$. Here, $\alpha=(\alpha_1,\ldots,\alpha_n)\in\mathbb{N}^n$ and $\mathbf{z}^\alpha=z_1^{\alpha_1}\cdots z_n^{\alpha_n}$. We assume that $f(\mathbf{0})=0$ and we write $V:=f^{-1}(0)$ for the hypersurface in $U\subseteq \mathbb{C}^n$ defined by $f$.

%%%%%%%%%%%%%%%%%%%%%%%%%%%%%%%%%%%%%%%%%%%%%%%%%%%%%%%%%%%%%%%%%%%%%%%%%%%%%%%%%%%%%%%%%%%%%
\subsection{The A'Campo formula}\label{ACformula}
%%%%%%%%%%%%%%%%%%%%%%%%%%%%%%%%%%%%%%%%%%%%%%%%%%%%%%%%%%%%%%%%%%%%%%%%%%%%%%%%%%%%%%%%%%%%%
The starting point for the results of \cite{V,O2} mentioned above is another famous formula for the zeta-function of the monodromy due to A'Campo \cite{A}. In this subsection, we recall this formula in a slightly more general form as given by the second named author in \cite{O1}.
 
Let $F$ denote the Milnor fibre of the Milnor fibration of $f$ at $\mathbf{0}$ and let $h\colon F\to F$ be the associated monodromy map. The \emph{zeta-function $\zeta_{f,\mathbf{0}}(t)$ of the monodromy associated with the Milnor fibration of $f$ at $\mathbf{0}$} is defined by
\begin{equation*}
\zeta_{f,\mathbf{0}}(t):=\prod_{i=0}^{n-1} P_i(t)^{(-1)^{i+1}}.
\end{equation*}
Here, $P_i(t):=\det(\mbox{Id}-t\cdot h_{*i})$, where $h_{*i}\colon H_i(F;\mathbb{Q})\to H_i(F;\mathbb{Q})$ is the homomorphism induced by $h$ on the $i$th homology group of $F$ with coefficients in $\mathbb{Q}$. Note that in the special case of isolated singularities, the fibre $F$ is $(n-2)$-connected, so that
\begin{equation*}
\zeta_{f,\mathbf{0}}(t):=(1-t)^{-1} P_{n-1}(t)^{(-1)^{n}}
\end{equation*}
and the Milnor number $\mu_{\mathbf{0}}(f)$ of $f$ at $\mathbf{0}$ satisfies the relation
\begin{equation}\label{lprrmnedzf}
-1+(-1)^n\mu_{\mathbf{0}}(f)=\deg\zeta_{f,\mathbf{0}}(t).
\end{equation}
(Here, by definition, the degree of a rational function $R(t)=p(t)/q(t)$ is the number $\deg R(t):=\deg p(t)-\deg q(t)$.)

Now, assume we are given a \emph{good resolution of the function $f$}, that is, a proper holomorphic map $\hat{\pi}\colon X\to U$ from a (complex) analytic manifold $X$ of dimension $n$ to the neighbourhood $U$ satisfying the following two conditions:
\begin{enumerate}
\item
the restriction $\hat{\pi}\colon X\setminus \hat{\pi}^{-1}(V)\to U\setminus V$ is biholomorphic;
\item
if $\hat{\pi}^{-1}(V)=E_1\cup\cdots\cup E_r\cup E_{r+1}\cup\cdots\cup E_{r+m}$ and $\widetilde{V}=E_{r+1}\cup\cdots\cup E_{r+m}$ denote the irreducible decompositions of the total transform $\hat{\pi}^{-1}(V)$ and of the strict transform $\widetilde{V}$ of $V$ by $\hat{\pi}$, respectively, then the $E_i$'s ($1\leq i\leq r+m$) are non-singular and $\hat{\pi}^{-1}(V)$ has only normal crossing singularities. 
\end{enumerate}
The second condition means that for any $\mathbf{p}\in \hat{\pi}^{-1}(V)$, if $I$ is the set of indexes $i$ ($1\leq i\leq r+m$) for which $\mathbf{p}\in E_i$, then $|I|\leq n$ and there is an analytic coordinate chart $(U_{\mathbf{p}},\mathbf{x}:=(x_1,\ldots,x_n))$ of $X$ at $\mathbf{p}$ together with an injective map $\nu\colon I\to \{1,\ldots,n\}$ such that, in this chart, $E_i$ is given by $x_{\nu(i)}=0$ for all $i\in I$.

Now, for all $1\leq i\leq r$, let $m_i$ denote the multiplicity along $E_i$ of the pull-back function $\hat{\pi}^*f$ of $f$ by $\hat\pi$, and let 
\begin{equation*}
E'_i:=\hat{\pi}^{-1}(0)\cap \Bigg(E_i\Bigg\backslash \widetilde{V}\cup \bigcup_{\substack{1\leq j\leq r\\ j\not=i}} E_j\Bigg).
\end{equation*}
Then the A'Campo--Oka formula for the zeta-function $\zeta_{f,\mathbf{0}}(t)$ says that
\begin{equation}\label{ACampoformula}
\zeta_{f,\mathbf{0}}(t)=\prod_{i=1}^r (1-t^{m_i})^{-\chi(E'_i)}
\end{equation}
where $\chi(E'_i)$ is the Euler--Poincar\'e characteristic of $E'_i$ (see \cite[Th\'eor\`eme 3]{A} and \cite[Chapter~I, Theorem (5.2)]{O1}). 
Let us highlight that this formula holds for possibly non-isolated singularities. 

\begin{remark}[see\mbox{\cite[Proposition 11]{O3}}]\label{remark-multetmi}
If $\mbox{mult}_{\mathbf{0}}(f)$ is the multiplicity of $f$ at $\mathbf{0}$, then for each $1\leq i\leq r$ we have $m_i\geq \mbox{mult}_{\mathbf{0}}(f)$.
\end{remark}

To state the formulas by Varchenko and Oka about Newton non-degenerate and almost Newton non-degenerate functions, we first need to recall the notions of dual Newton diagram and toric modification. This is done in \S\S \ref{dnd} and \ref{TM} below. The formulas by Varchenko and Oka are given in \S \ref{subsect-VF} and in \S \ref{subsect-OF} respectively.

%%%%%%%%%%%%%%%%%%%%%%%%%%%%%%%%%%%%%%%%%%%%%%%%%%%%%%%%%%%%%%%%%%%%%%%%%%%%%%%%%%%%%%%%%%%%%
\subsection{Dual Newton diagram}\label{dnd}
%%%%%%%%%%%%%%%%%%%%%%%%%%%%%%%%%%%%%%%%%%%%%%%%%%%%%%%%%%%%%%%%%%%%%%%%%%%%%%%%%%%%%%%%%%%%%
Here, we recall the notion of dual Newton diagram. For details, we refer the reader to \cite[Chapter II, \S 1 and Chapter III, \S 3]{O1}.

Let $M$ be the lattice of Laurent monomials in the variables $z_1,\ldots,z_n$, and let $W$ be the (dual) lattice of (integral) weights on these variables, that is, the lattice of weight functions $\mathbf{w}\colon \{z_1,\ldots,z_n\}\to\mathbb{Z}$. Let 
\begin{equation*}
M_{\mathbb{R}}:=M\otimes_{\mathbb{Z}}\mathbb{R}
\quad\mbox{and}\quad
W_{\mathbb{R}}:=W\otimes_{\mathbb{Z}}\mathbb{R}
\end{equation*}
be the corresponding real vector spaces of dimension $n$. Hereafter, we identify these spaces with $\mathbb{R}^n$, and to avoid any confusion, we denote the vectors in $M_{\mathbb{R}}$ (respectively, in $W_{\mathbb{R}}$) by row vectors (respectively, by column vectors). So, in particular, a monomial $\mathbf{z}^\alpha=z_1^{\alpha_1}\cdots z_n^{\alpha_n}\in M$ is identified with the integral row vector $\alpha=(\alpha_1,\ldots,\alpha_n)$ while a weight $\mathbf{w}\in W$ is identified with the integral column vector ${}^t(w_1,\ldots,w_n):={}^t(\mathbf{w}(z_1),\ldots,\mathbf{w}(z_n))$. Define $M^+$ (respectively, $W^+$) as the set of all ``non-negative'' row vectors $(\alpha_1,\ldots,\alpha_n)\in M$ (respectively, all ``non-negative'' column vectors ${}^t(w_1,\ldots,w_n)\in W$) \textemdash\ that is, $\alpha_i\geq 0$ and $w_i\geq 0$ for all $1\leq i\leq n$. Define $M_{\mathbb{R}}^+$ and $W_{\mathbb{R}}^+$ similarly. Again, hereafter we identify $M_{\mathbb{R}}^+$ and $W_{\mathbb{R}}^+$ with 
\begin{equation*}
\mathbb{R}^n_{\geq 0}:=\{\mathbf{x}=(x_1,\ldots,x_n)\in\mathbb{R}^n\mid x_i\geq 0 \mbox{ for all } 1\leq i\leq n\}. 
\end{equation*}
The elements of $W_{\mathbb{R}}^+$ are called \emph{weight vectors}, and an integral weight vector ${}^t(w_1,\ldots,w_n)\in W^+\setminus\{\mathbf{0}\}$ is called \emph{primitive} if $\gcd(w_{1},\ldots,w_{n})=1$. 

Clearly, the Newton polyhedron $\Gamma_{\! +}(f)$ and the Newton boundary $\Gamma(f)$ of $f$ at $\mathbf{0}$ with respect to the coordinates $\mathbf{z}=(z_1,\ldots,z_n)$ can be viewed as subspaces of $M_{\mathbb{R}}^+$. We recall that $\Gamma_{\! +}(f)$ (or $\Gamma_{\! +}(f;\mathbf{z})$ when we need to emphasize the coordinates) is defined as the convex hull in $\mathbb{R}^n_{\geq 0}$ of the set
\begin{equation*}
\bigcup_{\alpha,\, a_{\alpha}\not=0}(\alpha+\mathbb{R}^n_{\geq 0})
\end{equation*}
while $\Gamma(f)$ is the union of the compact faces of $\Gamma_{\! +}(f)$.

Let us also recall that a \emph{convex polyhedral cone} $\sigma\subseteq W_{\mathbb{R}}^+$ is a set of the form
\begin{equation}\label{def-cpc}
\sigma=C(\mathbf{w}_1,\ldots,\mathbf{w}_k):=\bigg\{\sum_{i=1}^k \lambda_i \mathbf{w}_i\in W_{\mathbb{R}}^+\mid \lambda_i\in\mathbb{R}_{\geq 0}\mbox{ for all } 1\leq i\leq k\bigg\}.
\end{equation}
The vectors $\mathbf{w}_i\in W_{\mathbb{R}}^+$ that appear in \eqref{def-cpc} are called  \emph{generators} of $\sigma$. If they can be taken in $W^+$, then $\sigma$ is said to be \emph{rational}. 
Any rational convex polyhedral cone $\sigma\subseteq W_{\mathbb{R}}^+$ can be uniquely written as $\sigma=C(\mathbf{w}_1,\ldots,\mathbf{w}_k)$, where the $\mathbf{w}_i$'s are primitive and $k$ is minimal among all possible such expressions (i.e., $\mathbf{w}_i\notin C(\mathbf{w}_1,\ldots,\mathbf{w}_{i-1},\mathbf{w}_{i+1},\ldots,\mathbf{w}_k)$ for each $1\leq i\leq k$). Hereafter we always assume that  cones are generated by the minimal generators. The dimension of a convex polyhedral cone $\sigma\subseteq W_{\mathbb{R}}^+$ is its Euclidean dimension. A rational convex polyhedral cone $\sigma=C(\mathbf{w}_1,\ldots,\mathbf{w}_k)$ is said to be \emph{simplicial} if $\mathbf{w}_1,\ldots,\mathbf{w}_k$ are linearly independent over $\mathbb{R}$; it is said to be \emph{regular} if $\mathbf{w}_1,\ldots,\mathbf{w}_k$ are primitive and can be completed in a basis of the lattice $\mathbb{Z}^n$. 

A family $\Sigma^*$ of rational convex polyhedral cones of $W_{\mathbb{R}}^+$ is called a \emph{rational convex polyhedral cone subdivision} of $W_{\mathbb{R}}^+$ if $\Sigma^*$ is a finite complex\footnote{This means that any face of a cone of $\Sigma^*$ is also a cone of $\Sigma^*$ and the intersection of any cones $\sigma$ and $\sigma'$ of $\Sigma^*$ is a face of both $\sigma$ and $\sigma'$. We recall that if $\sigma=C(\mathbf{w}_1,\ldots,\mathbf{w}_k)$ and $I=\{i_1,\ldots,i_m\}\subseteq\{1,\ldots,n\}$, then the cone $\sigma_I=C(\mathbf{w}_{i_1},\ldots,\mathbf{w}_{i_m})$ is a face of $\sigma$ if there exists a hyperplane $H$ of $W_{\mathbb{R}}$ through $\mathbf{w}_{i_1},\ldots,\mathbf{w}_{i_m}$ such that the other vertices of $\sigma$ are located in the same connected component of $W_{\mathbb{R}}^+\setminus H$.} such that 
\begin{equation*}
W_{\mathbb{R}}^+=\bigcup_{\sigma\in\Sigma^*}\sigma.
\end{equation*} 
Note that since we are dealing with cones having the origin as vertex, we can identify any rational convex polyhedral cone subdivision with its projection on the ``hyperplane'' of $W_{\mathbb{R}}^+\equiv\mathbb{R}^n_{\geq 0}$ defined by the equation $x_1+\cdots+x_n=1$. 
A rational convex polyhedral cone subdivision $\Sigma^*$ of $W_{\mathbb{R}}^+$ is called a \emph{simplicial cone subdivision} of $W_{\mathbb{R}}^+$ if every cone $\sigma\in\Sigma^*$ is simplicial. A simplicial cone subdivision $\Sigma^*$ of $W_{\mathbb{R}}^+$ is called a \emph{regular simplicial cone subdivision} of $W_{\mathbb{R}}^+$ if every simplicial cone $\sigma\in\Sigma^*$ is regular. Finally, a \emph{vertex} of a regular simplicial cone subdivision $\Sigma^*$ is a primitive weight vector which generates a $1$-dimensional cone of $\Sigma^*$.

Now, for any weight vector $\mathbf{w}={}^t(w_1,\ldots,w_n)\in W_{\mathbb{R}}^+$, write $d(\mathbf{w};f)$ for the minimal value of the restriction to $\Gamma_{\! +}(f)$ of the canonical linear map $\ell_{\mathbf{w}}$ defined by 
\begin{equation}\label{linfct}
\ell_{\mathbf{w}}(\alpha):=\sum_{i=1}^n \alpha_i w_i,
\end{equation}
 and put 
\begin{equation*}%\label{faceaw}
\Delta(\mathbf{w};f):=\{\alpha\in\Gamma_{\! +}(f)\mid \ell_{\mathbf{w}}(\alpha)=d(\mathbf{w};f)\}.
\end{equation*} 
Clearly, $\Delta(\mathbf{w};f)$ is a face of $\Gamma_{\! +}(f)$. It is a compact face (i.e., it is contained in the Newton boundary $\Gamma(f)$) if and only if $\mathbf{w}$ is a ``positive'' weight vector (i.e., if $w_i>0$ for each $i$). 

In order to define the dual Newton diagram, we consider on $W_{\mathbb{R}}^+$ the equivalence relation $\sim$ defined for any $\mathbf{w}, \mathbf{w}'\in W_{\mathbb{R}}^+$ as follows:
\begin{equation*}
\mathbf{w}\sim \mathbf{w}' \Leftrightarrow \Delta(\mathbf{w};f)=\Delta(\mathbf{w}';f).
\end{equation*}
For any face $\Delta\subseteq\Gamma_{\! +}(f)$, there is an equivalence class $\Delta^*$ which is defined by 
\begin{equation*}
\Delta^*:=\{\mathbf{w}\in W^+_{\mathbb{R}}\mid \Delta(\mathbf{w};f)=\Delta\}.
\end{equation*} 
For each $(n-1)$-dimensional face $\Delta_0\subseteq\Gamma_{\! +}(f)$, there is a unique primitive weight vector $\mathbf{w}(\Delta_0)\in W^+\setminus\{\mathbf{0}\}$ such that $\Delta_0=\Delta(\mathbf{w}(\Delta_0);f)$, and for any face $\Delta\subseteq\Gamma_{\! +}(f)$, the corresponding equivalence class $\Delta^*$ is of the form
\begin{equation*}
\Delta^*=\bigg\{\sum_{i=1}^k \lambda_i \mathbf{w}(\Delta_i)\in W_{\mathbb{R}}^+\mid \lambda_i\in\mathbb{R}_{> 0}\mbox{ for all } 1\leq i\leq k\bigg\},
\end{equation*}
where the $\Delta_i$'s are the $(n-1)$-dimensional faces of $\Gamma_{\! +}(f)$ containing $\Delta$. The family 
\begin{equation*}
\{\Delta^*\}_{\Delta\subsetneq \Gamma_{\! +}(f)}
\end{equation*}
(where $\Delta$ runs over all proper faces of $\Gamma_{\! +}(f)$) is a partition of $W_{\mathbb{R}}^+\setminus \{\mathbf{0}\}\equiv\mathbb{R}^n_{\geq 0}\setminus\{\mathbf{0}\}$.

\begin{definition}
The \emph{dual Newton diagram} $\Gamma^*(f)$ (or $\Gamma^*(f;\mathbf{z})$) of $f$ at $\mathbf{0}$ with respect to the coordinates $\mathbf{z}=(z_1,\ldots,z_n)$ is the rational convex polyhedral cone subdivision of $W_{\mathbb{R}}^+$ given by the closures
\begin{equation*}
\bar\Delta^*=\bigg\{\sum_{i=1}^k \lambda_i \mathbf{w}(\Delta_i)\in W_{\mathbb{R}}^+\mid \lambda_i\in\mathbb{R}_{\geq 0}\mbox{ for all } 1\leq i\leq k\bigg\}
\end{equation*}
of the equivalence classes $\Delta^*$ associated with the relation $\sim$.
\end{definition}

The next definition concerns a class of regular simplicial cone subdivisions which will play a crucial role in what follows.

\begin{definition}\label{def-admisub}
A regular simplicial cone subdivision $\Sigma^*$ of $W_{\mathbb{R}}^+$ is said to be \emph{admissible} with respect to the dual Newton diagram $\Gamma^*(f)$ if $\Sigma^*$ is a regular simplicial cone subdivision of $\Gamma^*(f)$ (i.e., any cone $\sigma\in\Sigma^*$ is contained in a cone $\sigma'\in\Gamma^*(f)$).
\end{definition}

%%%%%%%%%%%%%%%%%%%%%%%%%%%%%%%%%%%%%%%%%%%%%%%%%%%%%%%%%%%%%%%%%%%%%%%%%%%%%%%%%%%%%%%%%%%%%
\subsection{Toric modification}\label{TM}
%%%%%%%%%%%%%%%%%%%%%%%%%%%%%%%%%%%%%%%%%%%%%%%%%%%%%%%%%%%%%%%%%%%%%%%%%%%%%%%%%%%%%%%%%%%%%
In this subsection, we briefly recall a standard construction, called ``toric modification'',  which is used in Varchenko's and Oka's formulas and which we will use hereafter too. Again for details, we refer the reader to \cite[Chapter II, \S 1]{O1}.

Let $\Sigma^*$ be a regular simplicial cone subdivision of $W_{\mathbb{R}}^+$. Associated with such a subdivision, we construct a (complex) manifold $X\equiv X(\Sigma^*)$ together with a map $\hat{\pi}\colon X\to\mathbb{C}^n$ as follows.
Let us denote by $\Sigma^*(n)$ the set of $n$-dimensional cones in $\Sigma^*$, and for each $\sigma=C(\mathbf{w}_1,\ldots,\mathbf{w}_n)\in\Sigma^*(n)$, let $\mathbb{C}^n_\sigma$ be the affine space of dimension $n$ with coordinates $\mathbf{y}_\sigma=(y_{\sigma,1},\dots,y_{\sigma,n})$ and let $\hat{\pi}_\sigma\colon \mathbb{C}^n_\sigma\to\mathbb{C}^n$ be the birational map defined by 
\begin{equation}\label{def-bmpsh}
\hat{\pi}_\sigma(\mathbf{y}_\sigma):=\bigg(\prod_{j=1}^n y_{\sigma,j}^{w_{1,j}},\dots, \prod_{j=1}^n y_{\sigma,j}^{w_{n,j}}\bigg),
\end{equation}
where ${}^t(w_{1,i},\ldots,w_{n,i}):=\mathbf{w}_i$ for $1\leq i\leq n$.
Now, on the disjoint union $\bigsqcup_{\sigma\in\Sigma^*(n)}\mathbb{C}^n_\sigma$, let us consider the equivalence relation $\approx$ defined for any $\mathbf{y}_{\sigma}\in\mathbb{C}^n_{\sigma}$ and $\mathbf{y}_{\sigma'}\in\mathbb{C}^n_{\sigma'}$ as follows: 
\begin{equation*}
\mathbf{y}_\sigma\approx\mathbf{y}_{\sigma'}
\Leftrightarrow\left\{
\mbox{\begin{minipage}{10cm}
the birational map $\hat{\pi}_{\sigma'}^{-1}\circ\hat{\pi}_\sigma\colon \mathbb{C}^n_\sigma\to\mathbb{C}^n_{\sigma'}$ is defined at $\mathbf{y}_{\sigma}$ and $\hat{\pi}_{\sigma'}^{-1}\circ\hat{\pi}_\sigma(\mathbf{y}_{\sigma})=\mathbf{y}_{\sigma'}$.
\end{minipage}}
\right.
\end{equation*}
The quotient space 
\begin{equation*}
X:=\bigsqcup_{\sigma\in\Sigma^*(n)}\mathbb{C}^n_\sigma \ \bigg / \approx
\end{equation*}
is a non-singular algebraic variety with coordinate charts $(\mathbb{C}^n_{\sigma},\mathbf{y}_\sigma)$, where $\sigma$ runs over all cones of $\Sigma^*(n)$ \textemdash\ usually these chart are called \emph{toric coordinate charts} \textemdash\ and the canonical map 
\begin{equation*}
\hat{\pi}\colon X\to\mathbb{C}^n, 
\end{equation*}
defined in each chart $(\mathbb{C}^n_{\sigma},\mathbf{y}_\sigma)$ by $\hat{\pi}([\mathbf{y}_{\sigma}]):=\hat{\pi}_\sigma(\mathbf{y}_{\sigma})$, is a proper birational morphism. (Here, $[\mathbf{y}_{\sigma}]$ denotes the class of $\mathbf{y}_{\sigma}$ with respect to the equivalence relation $\approx$.)
The variety $X$ constructed in this way is called the \emph{toric variety} associated with $\Sigma^*$ and the map $\hat{\pi}\colon X\to\mathbb{C}^n$ is called the \emph{toric modification} (or \emph{toric blowing-up}) associated with $\Sigma^*$.

%%%%%%%%%%%%%%%%%%%%%%%%%%%%%%%%%%%%%%%%%%%%%%%%%%%%%%%%%%%%%%%%%%%%%%%%%%%%%%%%%%%%%%%%%%%%%
\subsection{The Varchenko formula}\label{subsect-VF}
%%%%%%%%%%%%%%%%%%%%%%%%%%%%%%%%%%%%%%%%%%%%%%%%%%%%%%%%%%%%%%%%%%%%%%%%%%%%%%%%%%%%%%%%%%%%%
Throughout this subsection, we assume that $f$ is \emph{Newton non-degenerate} (i.e., for any face $\Delta\subseteq\Gamma(f)$, the face function $f_{\Delta}(\mathbf{z}):=\sum_{\alpha\in\Delta}a_\alpha\mathbf{z}^\alpha$ has no critical point in $\mathbb{C}^{*n}:=\{\mathbf{z}\in\mathbb{C}^n\mid z_1\cdots z_n\not=0\}$). On the other hand, we do not assume that $f$ is ``convenient'' (i.e., we do not assume that $\Gamma(f)$ intersects each coordinate axis), so that it may have a non-isolated singularity at the origin.  

Let $\Sigma^*$ be a regular simplicial cone subdivision of $W_{\mathbb{R}}^+$, and let $\hat{\pi}\colon X\to\mathbb{C}^n$ be the associated toric modification. 
Under the relation $\approx$, for any cones $\sigma=C(\mathbf{w}_1,\ldots,\mathbf{w}_n)$ and $\sigma'=C(\mathbf{w}'_1,\ldots,\mathbf{w}'_n)$ of $\Sigma^*(n)$ having a common vertex $\mathbf{w}$ of $\Sigma^*$, say, for instance, $\mathbf{w}:=\mathbf{w}_1=\mathbf{w}'_1$ (changing the orderings of $\mathbf{w}_1,\ldots,\mathbf{w}_n$ and of $\mathbf{w}'_1,\ldots,\mathbf{w}'_n$ if necessary), the divisors
\begin{equation*}
\hat E(\mathbf{w};\sigma):=\{\mathbf{y}_\sigma\in \mathbb{C}^{n}_\sigma \mid y_{\sigma,1}=0\} \quad\mbox{and}\quad \hat E(\mathbf{w};\sigma'):=\{\mathbf{y}_{\sigma'}\in \mathbb{C}^{n}_{\sigma'} \mid y_{\sigma',1}=0\}
\end{equation*}
glue together on 
\begin{equation*}
\{\mathbf{y}_\sigma\in \mathbb{C}^{n}_\sigma \mid y_{\sigma,1}=0,\, y_{\sigma,j}\not=0\mbox{ for } j\geq 2\}
\end{equation*}
 and on 
\begin{equation*}
\{\mathbf{y}_{\sigma'}\in \mathbb{C}^{n}_{\sigma'} \mid y_{\sigma',1}=0,\, y_{\sigma',j}\not=0\mbox{ for }  j\geq 2\},\footnote{More precisely, if $\sigma\cap\sigma'=C(\mathbf{w}_1,\ldots,\mathbf{w}_\ell)$ with $\mathbf{w}_i=\mathbf{w}'_i$ for all $1\leq i\leq \ell$ (again changing the orderings of $\mathbf{w}_1,\ldots,\mathbf{w}_n$ and of $\mathbf{w}'_1,\ldots,\mathbf{w}'_n$ if necessary), then the divisors $\hat E(\mathbf{w};\sigma)$ and $\hat E(\mathbf{w};\sigma')$ glue together on $\{\mathbf{y}_\sigma\in \mathbb{C}^{n}_\sigma \mid y_{\sigma,1}=0,\, y_{\sigma,j}\not=0\mbox{ for } j\geq \ell+1\}$ and on $\{\mathbf{y}_{\sigma'}\in \mathbb{C}^{n}_{\sigma'} \mid y_{\sigma',1}=0,\, y_{\sigma',j}\not=0\mbox{ for }  j\geq \ell+1\}$.}
\end{equation*}
so that for any vertex $\mathbf{w}$ of $\Sigma^*$ the canonical image in $X$ of the disjoint union of the $\hat E(\mathbf{w};\sigma)$'s for $\sigma\ni\mathbf{w}$ defines an irreducible divisor $\hat E(\mathbf{w})$ in $X$. 

Now, suppose that the subdivision $\Sigma^*$ is \emph{admissible} with respect to the dual Newton diagram $\Gamma^*(f)$ (see Definition \ref{def-admisub}). Then $\hat{\pi}\colon X\to\mathbb{C}^n$ is a good resolution of $f$. Let $\mbox{Vert}(\Sigma^*)$ denote the set of vertices of $\Sigma^*$. By \cite[Chapter III, Proposition (3.3)]{O1}, we may assume that $\Sigma^*$ is \emph{small}. In the special case where $f$ is monomial-factor free (i.e.,  the case where the factorization of $f$ into irreducible factors does not have any monomial factor), this means that whenever $f\vert_{\mathbb{C}^I}\not\equiv 0$, the cone $C(\mathbf{e}_{i_1},\ldots,\mathbf{e}_{i_q})$ is in $\Sigma^*$, where $\mathbf{e}_{i_1},\ldots,\mathbf{e}_{i_q}$ are all the elements in $\{\mathbf{e}_{1},\ldots,\mathbf{e}_{n}\}$ whose index $i_j$ is not in $I$. (Here, $\mathbf{e}_i:={}^t(0,\ldots,0,1,0,\ldots,0)$ with $1$ at the $i$th place.) Equivalently, for any vertex $\mathbf{w}\in\mbox{Vert}(\Sigma^*)$ different from $\mathbf{e}_1,\ldots,\mathbf{e}_n$, we have $d(\mathbf{w};f)>0$.  If $f$ is not monomial-factor free, then $f$ is written as $f=M\cdot f'$ where $M$ is a monomial and $f'$ is monomial-factor free, and in this case we say that $\Sigma^*$ is small for $f$ if it is small for the monomial-factor free function $f'$. This definition makes sense as $\Gamma^*(f)=\Gamma^*(f')$. Note that if $f$ is not monomial-factor free and if $\Sigma^*$ is small for $f$, then we still have $d(\mathbf{w};f)>0$ for any vertex $\mathbf{w}\in\mbox{Vert}(\Sigma^*)$ different from $\mathbf{e}_1,\ldots,\mathbf{e}_n$.

Consider the following set of vertices
\begin{equation*}
\mathcal{V}^{+}(f):=\{\mathbf{w}\in\mbox{Vert}(\Sigma^*)\mid d(\mathbf{w};f)>0\}.
\end{equation*}
 Then, by \cite[Chapter III, Theorem (3.4)]{O1}, we have 
\begin{equation*}
\hat{\pi}^{-1}(V)=\widetilde{V}\cup\bigcup_{\mathbf{w}\in\mathcal{V}^{+}(f)}\hat E (\mathbf{w})
\end{equation*}
and the multiplicity of $\hat{\pi}^*f$ along $\hat E (\mathbf{w})$ is $d(\mathbf{w};f)$. For each $\mathbf{w}\in\mathcal{V}^{+}(f)$, put
\begin{equation*}
\hat E'(\mathbf{w}):=\hat{\pi}^{-1}(\mathbf{0})\cap \Bigg(\hat E (\mathbf{w})\Bigg\backslash \widetilde{V}\cup \bigcup_{\substack{\mathbf{w}'\in\mathcal{V}^{+}(f)\\ \mathbf{w}'\not=\mathbf{w}}} \hat E (\mathbf{w}')\Bigg).
\end{equation*}
Then, by the A'Campo--Oka formula \eqref{ACampoformula}, the zeta-function $\zeta_{f,\mathbf{0}}(t)$ of the monodromy of the Milnor fibration of $f$ at~$\mathbf{0}$ is then given by
\begin{equation}\label{ACinV}
\zeta_{f,\mathbf{0}}(t)=\prod_{\mathbf{w}\in\mathcal{V}_{+}(f)} (1-t^{d(\mathbf{w};f)})^{-\chi(\hat E'(\mathbf{w}))}.
\end{equation}
Here, the main difficulty is to compute $\chi(\hat E'(\mathbf{w}))$. Under the Newton non-degeneracy assumption, in \cite{V}, Varchenko showed that \eqref{ACinV} can be rewritten as
\begin{equation}\label{varchenkoformula}
\zeta_{f,\mathbf{0}}(t)=\prod_{I\in \mathcal{I}} \zeta_I(t)
\quad\mbox{with}\quad \zeta_I(t):=\prod_{\mathbf{w}\in P^I}(1-t^{d(\mathbf{w};f^I)})^{-\chi(\mathbf{w})},
\end{equation}
where $\mathcal{I}$ is the set of all non-empty subsets $I\subseteq\{1,\ldots,n\}$ such that $f^I:=f\vert_{\mathbb{C}^I}\not\equiv 0$ and $P^I$ is the set of primitive positive weight vectors in $W^{+I}$ which correspond to the maximal dimensional faces of $\Gamma(f^I)$, that is, the set of vectors $\mathbf{w}={^t(w_1,\ldots,w_n)}\in W$ such that
\begin{equation*}
w_i>0\mbox{ for }i\in I,\ w_i=0\mbox{ for }i\notin I 
\mbox{ and }\dim \Delta(\mathbf{w};f^I)=|I|-1.
\end{equation*}
(Note that $P^I$ corresponds bijectively to the $(|I|-1)$-dimensional faces of $\Gamma(f^I)$.)
The number $\chi(\mathbf{w})$ in \eqref{varchenkoformula} is defined by
\begin{equation*}
\chi(\mathbf{w}):=(-1)^{|I|-1}\, |I|!\, \mbox{Vol}_{|I|}\big(\mbox{Cone}(\Delta(\mathbf{w};f^I),\mathbf{0}^I)\big) / d(\mathbf{w};f^I),
\end{equation*}
where $\mbox{Cone}(\Delta(\mathbf{w};f^I),\mathbf{0}^I):=\{\lambda\mathbf{x}\mid \mathbf{x}\in \Delta(\mathbf{w};f^I),\, 0\leq\lambda\leq 1\}$ is the closed cone over $\Delta(\mathbf{w};f^I)$ with the origin $\mathbf{0}^I$ of $\mathbb{C}^I$ as vertex and where $\mbox{Vol}_{|I|}$ is the $|I|$-dimensional Euclidean volume. (Here, $|I|$ denotes the cardinality of $I$.)

Again, like for \eqref{ACampoformula}, let us highlight that the formula \eqref{varchenkoformula} do hold true for possibly non-isolated singularities. In the special case where $f$ has an isolated singularity at $\mathbf{0}$, the Milnor number $\mu_{\mathbf{0}}(f)$ of $f$ at $\mathbf{0}$ satisfies the following relation: 
\begin{equation}\label{var-relmilnornumber}
-1+(-1)^n\mu_{\mathbf{0}}(f)=\deg \zeta_{f,\mathbf{0}}(t)=\sum_{I\in\mathcal{I}} \deg\zeta_I(t)=\sum_{I\in\mathcal{I}}\sum_{\mathbf{w}\in P^I} -d(\mathbf{w};f^I)\, \chi(\mathbf{w}).
\end{equation}

In \cite{O2}, the second named author extended the Varchenko formula \eqref{varchenkoformula} to a larger class of functions called ``almost Newton non-degenerate functions.'' This class includes all Newton non-degenerate functions. The following two subsections are devoted to this generalization which will be useful for our purpose later.

%%%%%%%%%%%%%%%%%%%%%%%%%%%%%%%%%%%%%%%%%%%%%%%%%%%%%%%%%%%%%%%%%%%%%%%%%%%%%%%%%%%%%%%%%%%%%
\subsection{Almost Newton non-degenerate functions}\label{subsect-anndf}
%%%%%%%%%%%%%%%%%%%%%%%%%%%%%%%%%%%%%%%%%%%%%%%%%%%%%%%%%%%%%%%%%%%%%%%%%%%%%%%%%%%%%%%%%%%%%
This class of functions, introduced by the second named author in \cite{O2}, is defined as follows. From now on, let us suppose that $f$ is \emph{convenient}. Again, pick a regular simplicial cone subdivision $\Sigma^*$ of $W^+_{\mathbb{R}}$ which is admissible with respect to the dual Newton diagram $\Gamma^*(f)$, and consider the toric modification $\hat{\pi}\colon X\to \mathbb{C}^n$ associated with $\Sigma^*$. As in \S\ref{subsect-VF}, by \cite[Chapter III, Proposition~(3.3)]{O1}, we may assume that $\Sigma^*$ is small.
Let $\mathcal{M}$ denote the set of maximal dimensional faces of $\Gamma(f)$ (i.e., the faces of dimension $n-1$), and let $\mathcal{M}_0$ be the subset of $\mathcal{M}$ consisting of the faces $\Delta$ for which the face function $f_\Delta(\mathbf{z}):=\sum_{\alpha\in\Delta}a_\alpha\mathbf{z}^\alpha$ is \emph{Newton degenerate} (i.e., $f_{\Delta}$ has critical points in $\mathbb{C}^{*n}$).

\begin{definition}\label{defweaklyowanndf}
We say that $f$ is \emph{weakly almost Newton non-degenerate} if for any face $\Delta\subseteq\Gamma(f)$ the following two conditions hold true:
\begin{enumerate}
\item
if $\Delta\in\mathcal{M}\setminus\mathcal {M}_0$ or if $\dim\Delta\leq n-2$, then $f$ is Newton non-degenerate on $\Delta$ (i.e., the face function $f_\Delta$ has no critical point in $\mathbb{C}^{*n}$); 
\item
if $\Delta\in \mathcal{M}_0$, then the restriction $f_{\Delta}\colon\mathbb{C}^{*n}\to \mathbb{C}$ has a finite number of $1$-dimensional critical loci \textemdash\ which are $\mathbb{C}^*$-orbits of (some) elements in $\mathbb{C}^{*n}$ with respect to the associated $\mathbb{C}^*$-action defined by 
\begin{equation*}
(t,\mathbf{z})\in\mathbb{C}^*\times\mathbb{C}^{*n}\mapsto (t^{w_1}z_1,\ldots,t^{w_n}z_n)\in\mathbb{C}^{*n},
\end{equation*} 
where $(w_1,\dots,w_n)=\mathbf{w}$ is a weight vector such that $\Delta(\mathbf{w};f)=\Delta$.
\end{enumerate}
\end{definition}

\emph{Hereafter, we suppose that $f$ is weakly almost Newton non-degenerate.}
Take a face $\Delta\in\mathcal{M}_0$, and consider primitive weight vectors $\mathbf{w}_1,\dots,\mathbf{w}_n$ such that $\Delta(\mathbf{w}_1;f)=\Delta$ and $\sigma:=C(\mathbf{w}_1,\ldots,\mathbf{w}_n)\in\Sigma^*(n)$ (i.e., $\sigma$ is a maximal dimensional (regular simplicial) cone of $\Sigma^*$). (We recall that any cone of $\Sigma^*$ can be uniquely written in this form and that $\mathbf{w}_1$ is uniquely determined by the face $\Delta$.) Let $(\mathbb{C}^n_\sigma,\mathbf{y}_\sigma)$ be the toric coordinate chart corresponding to $\sigma$, and let $\tilde f_\sigma(\mathbf{y}_\sigma)$ be the function defined by the equality
\begin{align}\label{defeqforexcdiv1}
\hat{\pi}^*f(\mathbf{y}_\sigma)\equiv f(\hat{\pi}_\sigma(\mathbf{y}_\sigma)) =
\tilde f_\sigma(\mathbf{y}_\sigma)\prod_{i=1}^n y_{\sigma,i}^{d(\mathbf{w}_i;f)},
\end{align}
that is,
\begin{align}\label{defeqforexcdiv2}
\tilde f_\sigma(\mathbf{y}_\sigma)=\sum_{\alpha}a_\alpha\, y_{\sigma,1}^{\ell_{\mathbf{w}_1}(\alpha)-d(\mathbf{w}_1;f)}\cdots\, y_{\sigma,n}^{\ell_{\mathbf{w}_n}(\alpha)-d(\mathbf{w}_n;f)}
\end{align}
where $\mathbf{w}_i={}^t(w_{1,i},\ldots,w_{n,i})$ for $1\leq i\leq n$. (Here, as in \eqref{linfct}, for each $1\leq i\leq n$ we write $\ell_{\mathbf{w}_i}(\alpha):=\sum_{j=1}^n w_{j,i}\alpha_j$.)
In the chart $(\mathbb{C}^n_\sigma,\mathbf{y}_\sigma)$, the strict transform $\widetilde{V}$ of $V$ by $\hat{\pi}$ is given by the equation $\tilde f_\sigma(\mathbf{y}_\sigma)=0$ while the exceptional divisors 
\begin{equation*}
\hat{E}(\mathbf{w}_1)
\quad\mbox{and}\quad
E(\mathbf{w}_1):=\hat{E}(\mathbf{w}_1)\cap \widetilde{V}
\end{equation*}
 of $\hat{\pi}\colon X\to \mathbb{C}^n$ and of its restriction $\pi\colon\widetilde{V}\to V$, respectively, are given by the equations
\begin{equation*}
y_{\sigma,1}=0
\quad\mbox{and}\quad
y_{\sigma,1}=\tilde f_\sigma(0,y_{\sigma,2},\ldots,y_{\sigma,n})=0
\end{equation*}
respectively. Let us also define $\tilde f_{\mathbf{w}_1,\sigma}(\mathbf{y}_\sigma)$ by the following equation:
\begin{align*}
f_{\mathbf{w}_1}(\hat{\pi}_\sigma(\mathbf{y}_\sigma)) =
\tilde f_{\mathbf{w}_1,\sigma}(\mathbf{y}_\sigma)
\prod_{i=1}^n y_{\sigma,i}^{d(\mathbf{w}_i;f)},
\end{align*}
where $f_{\mathbf{w}_1}(\mathbf{z}):=\sum_{\alpha\in\Delta(\mathbf{w}_1;f)} a_\alpha\mathbf{z}^\alpha$ is the face function of $f$ with respect to the weight vector $\mathbf{w}_1$. Then, since $f_{\mathbf{w}_1}$ is weighted homogeneous with respect to the weight $\mathbf{w}_1$, we have $\ell_{\mathbf{w}_1}(\alpha)-d(\mathbf{w}_1;f)=0$, and therefore,
\begin{align*}
\tilde f_{\mathbf{w}_1,\sigma}(\mathbf{y}_\sigma)
& =\sum_{\alpha\in\Delta(\mathbf{w}_1;f)}a_\alpha\, y_{\sigma,2}^{\ell_{\mathbf{w}_2}(\alpha)-d(\mathbf{w}_2;f)}\cdots\, y_{\sigma,n}^{\ell_{\mathbf{w}_n}(\alpha)-d(\mathbf{w}_n;f)}\\
& =\tilde f_\sigma(0,y_{\sigma,2},\ldots,y_{\sigma,n}).
\end{align*}
In particular, $\tilde f_{\mathbf{w}_1,\sigma}(\mathbf{y}_\sigma)$ does not contain the variable $y_{\sigma,1}$, and in the chart $(\mathbb{C}^n_\sigma,\mathbf{y}_\sigma)$, the exceptional divisor $E(\mathbf{w}_1)$ is defined by the equations $y_{\sigma,1}=\tilde f_{\mathbf{w}_1,\sigma}(0,y_{\sigma,2},\ldots,y_{\sigma,n})=0$.
Thus, since $f$ is weakly almost Newton non-degenerate, it follows that the set $\mbox{Sing}(\Delta)$ of singular points of the hypersurface $E(\mathbf{w}_1)$ of $\hat E(\mathbf{w}_1)$ consists only in a finite number of points. (Indeed, $V(\tilde f_{\mathbf{w}_1,\sigma}\vert_{\{y_{\sigma,1}=0\}})\subseteq \{0\}\times\mathbb{C}^{n-1}$, and by identifying $\{0\}\times\mathbb{C}^{n-1}$ with $\mathbb{C}^{n-1}$, we see that the restriction of $\hat\pi_\sigma$ to $\mathbb{C}^*\times (V(\tilde f_{\mathbf{w}_1,\sigma}\vert_{\{y_{\sigma,1}=0\}})\cap\mathbb{C}^{*(n-1)})$ gives an isomorphism
\begin{equation*}
\mathbb{C}^*\times (V(\tilde f_{\mathbf{w}_1,\sigma}\vert_{\{y_{\sigma,1}=0\}})\cap\mathbb{C}^{*(n-1)})\overset{\sim}{\longrightarrow}
V(f_{\mathbf{w}_1})\cap\mathbb{C}^{*n}
\end{equation*}
(as usual, $V(f_{\mathbf{w}_1}):=f_{\mathbf{w}_1}^{-1}(0)$ and similarly for $V(\tilde f_{\mathbf{w}_1,\sigma}\vert_{\{y_{\sigma,1}=0\}})$); then the assertion follows from the weakly almost Newton non-degeneracy which says that the singular locus of the right-hand side of this isomorphism is $1$-dimensional.)
Pick a point $\mathbf{p}\in\mbox{Sing}(\Delta)$. In $\mathbb{C}^n_\sigma$, the coordinates of $\mathbf{p}$ are of the form $(0,p_2,\ldots,p_n)$. An analytic coordinate chart $(U_\mathbf{p},\mathbf{x}=(x_1,\dots,x_n))$ of $X$ at $\mathbf{p}$ is called \emph{admissible} (with respect to the cone $\sigma$) if $x_1=y_{\sigma,1}$ and $(x_2,\ldots,x_n)$ is an analytic coordinate change of $(y_{\sigma,2},\ldots,y_{\sigma,n})$. (In many cases, we can take $x_i=y_{\sigma,i}-p_{i}$ for $2\leq i\leq n$.) 

\begin{definition}\label{defowanndf}
We say that the \emph{weakly} almost Newton non-degenerate function $f$ is \emph{almost Newton non-degenerate} if for any $\Delta\in\mathcal{M}_0$ and any $\mathbf{p}\in \mbox{Sing}(\Delta)$, there exists an admissible coordinate chart $(U_\mathbf{p},\mathbf{x})$ of $X$ at $\mathbf{p}$ such that the function ${\hat \pi}^*f(\mathbf{x})$ on $U_\mathbf{p}$ is Newton non-degenerate with respect to the coordinates~$\mathbf{x}$.
\end{definition}

The following is an important example of such a function. It will be useful for our purpose later.

\begin{example}\label{sect-prel-example}
Take $n=3$ and suppose that $f(z_1,z_2,z_3)$ is a reduced, convenient, homogeneous polynomial of degree $d$ such that the corresponding projective curve 
\begin{equation*}
C:=\{(z_1:z_2:z_3)\in\mathbb{P}^2\mid f(z_1,z_2,z_3)=0\}
\end{equation*}
has only Newton non-degenerate singularities in some suitable local coordinates (in particular, this is always the case if the curve has only ``simple'' singularities in the sense of Arnol'd \cite{Arnold}). Assume further that all these singular points are located in $z_1z_2z_3\not=0$. Then for any integers $m\geq 1$ and $1\leq k\leq 3$, the function
\begin{equation*}
g_k(z_1,z_2,z_3):=f(z_1,z_2,z_3)+z_k^{d+m}
\end{equation*}
is almost Newton non-degenerate.
\end{example}

\begin{proof} 
To simplify, let us assume that $k=1$ and write $g(\mathbf{z})$ instead of $g_1(\mathbf{z})$. (Of course, the argument is completely similar for the other values of $k$.)
In the situation of Example~\ref{sect-prel-example}, the dual Newton diagram $\Gamma^*(f)$ has a single positive vertex $\mathbf{w}:={}^t(1,1,1)$ and $\Gamma^*(f)$ is already a regular simplicial cone subdivision of $W^+_{\mathbb{R}}$ with vertices $\mbox{Vert}(\Gamma^*(f))=\{\mathbf{w},\mathbf{e}_i\, (1\leq i\leq 3)\}$. The corresponding toric modification $\hat\pi\colon X\to\mathbb{C}^3$ is nothing but the usual point blowing-up at the origin. It has three canonical toric coordinate charts corresponding to the cones $\sigma:=C(\mathbf{w},\mathbf{e}_2,\mathbf{e}_3)$, $\sigma':=C(\mathbf{w},\mathbf{e}_1,\mathbf{e}_3)$ and $\sigma'':=C(\mathbf{w},\mathbf{e}_1,\mathbf{e}_2)$. In the chart $(\mathbb{C}^n_\sigma,\mathbf{y}_\sigma=(y_{\sigma,1},y_{\sigma,2},y_{\sigma,3}))$, the pull-back of the functions $f$ and $g$ by $\hat \pi$ are given by
\begin{align*}
& \hat\pi^*f(\mathbf{y}_\sigma)=y_{\sigma,1}^d \cdot \tilde f_\sigma(\mathbf{y}_\sigma)
=y_{\sigma,1}^d \cdot f(1,y_{\sigma,2},y_{\sigma,3})\quad\mbox{and}\\
& \hat\pi^*g(\mathbf{y}_\sigma)=y_{\sigma,1}^d \cdot \tilde g_\sigma(\mathbf{y}_\sigma)
=y_{\sigma,1}^d \cdot (\tilde f_\sigma(\mathbf{y}_\sigma)+y_{\sigma,1}^m)
=y_{\sigma,1}^d \cdot (f(1,y_{\sigma,2},y_{\sigma,3})+y_{\sigma,1}^m)
\end{align*}
respectively (see \eqref{defeqforexcdiv1} and \eqref{defeqforexcdiv2}).
The exceptional divisor $E_g(\mathbf{w})$ of the restriction of $\hat \pi$ to the strict transform of $g^{-1}(0)$ is a curve in $\hat E(\mathbf{w})$ which is defined in $\hat E(\mathbf{w})$ by the equation $\tilde g_\sigma(0,y_{\sigma,2},y_{\sigma,3})=0$, that is, by the equation $f(1,y_{\sigma,2},y_{\sigma,3})=0$. Since the singularities of $C$ are Newton non-degenerate for some suitable local coordinates, for each singular point $\mathbf{p}$ of $E_g(\mathbf{w})$ there is an admissible chart $(U_{\mathbf{p}},\mathbf{x}_\sigma=(y_{\sigma,1},x_{\sigma,2},x_{\sigma,3}))$ at $\mathbf{p}:=(0,p_2,p_3)$ such that, in this chart, the exceptional divisor $\hat E(\mathbf{w})$ is still given by $y_{\sigma,1}=0$ and the curve $E_g(\mathbf{w})$ is given in $\hat E(\mathbf{w})$ by an equation of the form $h(x_{\sigma,2},x_{\sigma,3})=0$, where $h$ is Newton non-degenerate with respect to the coordinates $(x_{\sigma,2},x_{\sigma,3})$. It follows that in the coordinates $\mathbf{x}_\sigma=(y_{\sigma,1},x_{\sigma,2},x_{\sigma,3})$, the pull-back of $g$, which is given by
\begin{equation*}
\hat\pi^*g(\mathbf{x}_\sigma)=y_{\sigma,1}^d \cdot 
(h(x_{\sigma,2},x_{\sigma,3})+y_{\sigma,1}^m),
\end{equation*}
is Newton non-degenerate. 
\end{proof}

%%%%%%%%%%%%%%%%%%%%%%%%%%%%%%%%%%%%%%%%%%%%%%%%%%%%%%%%%%%%%%%%%%%%%%%%%%%%%%%%%%%%%%%%%%%%%
\subsection{The Oka formula}\label{subsect-OF}
%%%%%%%%%%%%%%%%%%%%%%%%%%%%%%%%%%%%%%%%%%%%%%%%%%%%%%%%%%%%%%%%%%%%%%%%%%%%%%%%%%%%%%%%%%%%%
Now we have all the necessary material to recall the Oka formula \textemdash\ established in~\cite{O2} \textemdash\ for the zeta-function of the monodromy associated with the Milnor fibration at $\mathbf{0}$ of an almost Newton non-degenerate function. In fact, the proof given in \cite{O2} shows that the formula still holds true for a \emph{weakly} almost Newton non-degenerate function. \emph{So, hereafter in this subsection, we shall only assume that $f$ is weakly almost Newton non-degenerate.} We also continue with the same notation and assumptions as in \S\ref{subsect-anndf}.

For $0<\delta\ll\varepsilon$, we consider the (tubular) Milnor fibration 
\begin{equation}\label{Milnorfiboriginal}
f\colon U^*(\varepsilon,\delta)\to D_\delta^*, 
\end{equation}
of $f$ at $\mathbf{0}$, where 
\begin{equation*}
U^*(\varepsilon,\delta):=\{\mathbf{z}\in\mathbb{C}^n\,;\, 0<|f(\mathbf{z})|\leq \delta\mbox{ and }\Vert \mathbf z\Vert\leq \varepsilon\}
\quad\mbox{and}\quad 
D_\delta^*:=\{z\in\mathbb{C}\mid 0<|z|\leq\delta\}.
\end{equation*}
Clearly, since $\hat\pi$ is biholomorphic over $\mathbb{C}^n\setminus f^{-1}(0)$, this fibration can be ``lifted'' on $X$ as 
\begin{equation}\label{Milnorfiblifted}
\hat f:=f\circ\hat{\pi}\colon \hat U^*(\varepsilon,\delta):={\hat\pi}^{-1}(U^*(\varepsilon,\delta))\to D_\delta^*,
\end{equation}
 and the two fibrations \eqref{Milnorfiboriginal} and \eqref{Milnorfiblifted} are equivalent. For any face $\Delta\in\mathcal{M}_0$ and any point $\mathbf{p}\in\mbox{Sing}(\Delta)$, we also consider the local Milnor fibration 
\begin{equation}\label{localMilnorfib}
\hat f\colon \hat U^*_{\mathbf{p}}(\varepsilon',\delta)\to D_{\delta}^*
\end{equation}
 of the function $\hat f(\mathbf{x})=f\circ\hat\pi(\mathbf{x})=\hat{\pi}^*f(\mathbf{x})$ at $\mathbf{p}$, where $\delta\ll \min\{\varepsilon',\varepsilon\}$ and
\begin{equation*}
\hat U^*_{\mathbf{p}}(\varepsilon',\delta):=\{\mathbf{x}\in U_\mathbf{p}\, ; \, 0<|\hat f(\mathbf{x})|\leq \delta \mbox{ and } \|\mathbf{x}\|\leq \varepsilon'\}.
\end{equation*}
Here, $(U_\mathbf{p},\mathbf{x})$ is an admissible chart of $X$ at $\mathbf{p}$. We assume that $\delta$ is small enough, so that we can use the same $\delta$ for the local Milnor fibrations at points $\mathbf{p}\in\mbox{Sing}(\Delta)$ and for the lifted Milnor fibration \eqref{Milnorfiblifted}.
Now, we decompose the set $\hat U^*(\varepsilon,\delta)$ as
\begin{equation*}
\hat U^*(\varepsilon,\delta)=(\hat U^*(\varepsilon,\delta))' \ \cup \bigcup_{\substack{\mathbf{p}\in\mbox{\tiny Sing}(\Delta)\\ \Delta\in\mathcal{M}_0}} \hat U^*_{\mathbf{p}}(\varepsilon',\delta),
\end{equation*}
where the subset  $(\hat U^*(\varepsilon,\delta))'$ is defined by
\begin{equation*}
(\hat U^*(\varepsilon,\delta))':=\hat U^*(\varepsilon,\delta) \ \bigg\backslash \ 
\bigcup_{\substack{\mathbf{p}\in\mbox{\tiny Sing}(\Delta)\\ \Delta\in\mathcal{M}_0}} \hat U^*_{\mathbf{p}}(\varepsilon'',\delta)
\end{equation*}
 with $\varepsilon''$ a bit smaller than $\varepsilon'$, and we consider the corresponding decomposition of the lifted Milnor fibration \eqref{Milnorfiblifted}. We denote by $\zeta'(t)$ the zeta-function of the monodromy associated with the fibration $\hat f\colon (\hat U^*(\varepsilon,\delta))'\to D_{\delta}^*$, and as usual we write $\zeta_{\hat f,\mathbf{p}}(t)$ for the zeta-function of the monodromy associated with the local Milnor fibration of $\hat f$ at~$\mathbf{p}$. If $P$ and $P_0$ denote the sets of primitive positive weight vectors corresponding to $\mathcal{M}$ and $\mathcal{M}_0$, respectively, then, by \cite[Lemma~3 and Theorem 8]{O2}, the zeta-function $\zeta'(t)$ is given by
\begin{equation}\label{okaformula11}
\begin{aligned}
\prod_{\substack{I\subsetneq\{1,\ldots,n\}\\ I\not=\emptyset}}\zeta_I(t)\ & \times 
\prod_{\mathbf{w}\in P\setminus P_0} (1-t^{d(\mathbf{w};f)})^{-\chi(\mathbf{w})} \\
& \times\ \prod_{\mathbf{w}\in P_0} (1-t^{d(\mathbf{w};f)})^{-\chi(\mathbf{w})+(-1)^{n-1}\sum_{\mathbf{p}\in\mbox{\tiny Sing}(\Delta(\mathbf{w};f))}\, \mu_{\mathbf{p}}}
\end{aligned}
\end{equation}
and the zeta-function $\zeta_{f,\mathbf{0}}(t)$ of the monodromy associated with the Milnor fibration of $f$ at $\mathbf{0}$ is given by
\begin{equation}\label{okaformula12}
\zeta_{f,\mathbf{0}}(t)=\zeta'(t) \cdot\prod_{\substack{\mathbf{p}\in\mbox{\tiny Sing}(\Delta)\\ \Delta\in\mathcal{M}_0}} \zeta_{\hat f,\mathbf{p}}(t).
\end{equation}
In \eqref{okaformula11}, the factors $\zeta_I(t)$ and $(1-t^{d(\mathbf{w};f)})^{-\chi(\mathbf{w})}$ for $\mathbf{w}\in P\setminus P_0$ are as in Varchenko's formula \eqref{varchenkoformula} and $\mu_{\mathbf{p}}$ denotes the Milnor number at $\mathbf{p}$ of the hypersurface $E(\mathbf{w})$ of $\hat E(\mathbf{w})$.
The zeta-function $\zeta'(t)$ can be rewritten as 
\begin{equation}\label{okaformula13}
\zeta'(t)=\zeta_{f_s}(t)\cdot \prod_{\mathbf{w}\in P_0} (1-t^{d(\mathbf{w};f)})^{(-1)^{n-1}\sum_{\mathbf{p}\in\mbox{\tiny Sing}(\Delta(\mathbf{w};f))}\, \mu_{\mathbf{p}}},
\end{equation}
where $\{f_s(\mathbf{z})\}_{|s|<1}$ is an analytic deformation family of $f$ with respect to a parameter $s$ (i.e., for $s=0$ we have $f_0=f$) which is obtained from a small perturbation of the coefficients of the functions $f_\Delta$ for $\Delta\in P_0$ (so, in particular, we have $\Gamma(f_s)=\Gamma(f)$ for all $s$) such that $f_s$ is Newton non-degenerate for all $s\not=0$. Here, $\zeta_{f_s}(t)$ denotes the zeta-function $\zeta_{f_s,\mathbf{0}}(t)$ of the monodromy associated with the Milnor fibration of $f_s$ at $\mathbf{0}$ for $s\not=0$ (which is of course independent of such an $s$).

\begin{remark}
If, in addition, we assume that $f$ is almost Newton non-degenerate, 
then the zeta-functions $\zeta_{\hat f,\mathbf{p}}(t)$ in \eqref{okaformula12} can be computed using the Varchenko formula \eqref{varchenkoformula}. 
\end{remark}

\begin{remark}\label{rk28}
In the definition of almost Newton non-degenerate functions given in \cite{O2}, it is assumed that the function $\hat \pi^*f$ is ``pseudo convenient'' at $\mathbf{p}$ (i.e., of the form $\hat \pi^*f(\mathbf{z})=\mathbf{z}^{\alpha} h(\mathbf{z})$ in a neighbourhood of $\mathbf{p}$, where $h$ is a convenient function). However, this assumption is not necessary to obtain the formulas \eqref{okaformula11}--\eqref{okaformula13}. Indeed, the proof of these formulas uses the A'Campo--Oka formula \eqref{ACampoformula}. If $\hat \pi^*f$ is not ``pseudo convenient'' at $\mathbf{p}$, then the toric modification at $\mathbf{p}$ constructed in the course of the proof may contain non-compact exceptional divisors. However, the A'Campo--Oka formula still holds true in this case.
\end{remark}

Like for \eqref{ACampoformula} and \eqref{varchenkoformula}, the formulas \eqref{okaformula11}--\eqref{okaformula13} do hold true for possibly non-isolated singularities. Also, of course, in the special case where $f$ has an isolated singularity at $\mathbf{0}$, the Milnor number $\mu_{\mathbf{0}}(f)$ of $f$ at $\mathbf{0}$ can be computed from the zeta-function $\zeta_{f,\mathbf{0}}(t)$ using the following formula:
\begin{equation}\label{oka-relmilnornumber}
-1+(-1)^n\mu_{\mathbf{0}}(f)=\deg \zeta_{f,\mathbf{0}}(t) .
\end{equation}

%%%%%%%%%%%%%%%%%%%%%%%%%%%%%%%%%%%%%%%%%%%%%%%%%%%%%%%%%%%%%%%%%%%%%%%%%%%%%%%%%%%%%%%%%%%%%
\section{A shift formula for the Milnor number}\label{sect-sffmn}
%%%%%%%%%%%%%%%%%%%%%%%%%%%%%%%%%%%%%%%%%%%%%%%%%%%%%%%%%%%%%%%%%%%%%%%%%%%%%%%%%%%%%%%%%%%%%
Here, we prove the shift formula for the Milnor number mentioned in the introduction (see Theorem \ref{mt1} below). This formula will be used in Section \ref{sect-mszph} when studying $\mu^*$-Zariski pairs of surfaces. The main tool for the proof is the Oka formula \eqref{okaformula11}--\eqref{okaformula13} for the zeta-function of an almost Newton non-degenerate function.

Throughout this section, let $\mathbf{z}=(z_1,\ldots,z_n)$ be coordinates for $\mathbb{C}^n$ ($n\geq 2$), and let $f(\mathbf{z})=\sum_\alpha a_\alpha \mathbf{z}^\alpha$ be a convenient weighted homogeneous polynomial function with respect to a primitive weight vector $\mathbf{w}={}^t(w_{1},\ldots,w_{n})\in W^+\setminus\{\mathbf{0}\}$. Denote by $d$ the corresponding weighted degree of $f$, and as usual write $V:=f^{-1}(0)$ for the hypersurface of $\mathbb{C}^n$ defined by $f$. We assume that $f(\mathbf{0})=0$ and that the singular locus of $V$ is $1$-dimensional. We also suppose that for any proper subset $I\subsetneq\{1,\ldots,n\}$, the restriction $f^I:=f\vert_{\mathbb{C}^I}$ is Newton non-degenerate. In particular, since $f$ is weighted homogeneous and since the singular locus of $V$ is $1$-dimensional, this implies that $f$ is weakly almost Newton non-degenerate (see Definition \ref{defweaklyowanndf}).
Besides, since $f$ is convenient, the expression $f(\mathbf{z})=\sum_\alpha a_\alpha \mathbf{z}^\alpha$ necessarily contains (up to a coefficient) a monomial of the form $z_i^{d_i}$ for each $1\leq i\leq n$, and by the weighted homogeneity, we have $w_{i}d_i=d$. The convenience also implies that there exists a regular simplicial cone subdivision $\Sigma^*$ of $W^+_{\mathbb{R}}$ which is admissible with respect to the dual Newton diagram $\Gamma^*(f)$ and such that the vertices of $\Sigma^*$ different from the $\mathbf{e}_i$'s ($1\leq i\leq n$) are positive.
Let $\hat{\pi}\colon X\to \mathbb{C}^n$ be the toric modification associated with such a subdivision, and let $\sigma=C(\mathbf{w}_1,\ldots,\mathbf{w}_n)$ be an $n$-dimensional cone of $\Sigma^*$ with $\mathbf{w}_1=\mathbf{w}$. As above, we denote by $(\mathbb{C}^n_\sigma,\mathbf{y}_\sigma=(y_{\sigma,1},\ldots,y_{\sigma,n}))$ the corresponding toric coordinate chart of $X$. Since $f$ is weakly almost Newton non-degenerate, the hypersurface $E(\mathbf{w})$ of $\hat E(\mathbf{w})$  has only a finite number of singular points (see \S \ref{subsect-VF}). 

\begin{definition}\label{def-psND}
With the above assumptions, we say that $f$ is \emph{Newton pre-non-degenerate} if for each singular point $\mathbf{p}$ of $E(\mathbf{w})$, there exists an admissible coordinate chart $(U_\mathbf{p},\mathbf{x}_\mathbf{p}=(x_{\mathbf{p},1},\ldots,x_{\mathbf{p},n}))$ of $X$ at $\mathbf{p}$ with respect to the cone $\sigma$ \textemdash\ i.e., $x_{\mathbf{p},1}=y_{\sigma,1}$ and $\mathbf{x}'_\mathbf{p}:=(x_{\mathbf{p},2},\ldots,x_{\mathbf{p},n})$ is an analytic coordinate change of $(y_{\sigma,2},\ldots,y_{\sigma,n})$; in particular, $\mathbf{x}'_\mathbf{p}$ are analytic coordinates for $\hat E(\mathbf{w})$ \textemdash\ such that the defining function $h_\mathbf{p}(x_{\mathbf{p},2},\ldots,x_{\mathbf{p},n})$ of the hypersurface $E(\mathbf{w})$ is convenient and Newton non-degenerate with respect to the coordinates $\mathbf{x}'_\mathbf{p}$.
\end{definition} 

Now, for any $1\leq k\leq n$, consider the function 
\begin{equation*}
g_k(\mathbf{z}):=f(\mathbf{z})+z_k^{d_k+m},
\end{equation*} 
where $m$ is an integer $\geq 1$. The main result of this section says that under the Newton pre-non-degeneracy condition for $f$, the function $g_k$ is an almost Newton non-degenerate function with an isolated singularity at the origin and its Milnor number at $\mathbf{0}$ can be described in terms of the integers $d_1,\ldots,d_n$, the weight $\mathbf{w}$ and the Milnor numbers of the hypersurface singularities $(E(\mathbf{w}),\mathbf{p})$ for $\mathbf{p}$ running in the (finite) set of singular points of $E(\mathbf{w})$. More precisely, we have the following statement which generalizes \cite[Theorem~18]{O2} where the assertion is proved for homogeneous polynomials with $m=1$.

\begin{theorem}\label{mt1} 
Under the assumptions described in the preamble of the present section and if furthermore $f$ is Newton pre-non-degenerate, then for any integer $m\geq 1$ the polynomial function 
\begin{equation*}
g_k(\mathbf{z})=f(\mathbf{z})+z_k^{d_k+m}
\end{equation*}
 is an almost Newton non-degenerate function with an isolated singularity at the origin and its Milnor number $\mu_\mathbf{0}(g_k)$ at $\mathbf{0}$ is given by
\begin{equation}\label{edmt1}
\mu_\mathbf{0}(g_k)=\prod_{i=1}^n (d_i-1) + 
m\, w_k\, \mu^{\mbox{\tiny \emph{tot}}}.
\end{equation}
Here, $\mu^{\mbox{\tiny \emph{tot}}}:=\sum \mu_\mathbf{p}$ where the sum is taken over all points $\mathbf{p}$ contained in the (finite) set $\mbox{\emph{Sing}}(E(\mathbf{w}))$ consisting of the singular points of the hypersurface $E(\mathbf{w})$ and where $\mu_\mathbf{p}$ denotes the Milnor number of the hypersurface singularity $(E(\mathbf{w}),\mathbf{p})$.
\end{theorem}

\begin{proof}
To simplify, let us assume that $k=1$ and write $g(\mathbf{z})$ instead of $g_1(\mathbf{z})$. (Of course, the argument is completely similar for the other values of $k$.)
In the chart $(\mathbb{C}^n_\sigma,\mathbf{y}_\sigma=(y_{\sigma,1},\ldots,y_{\sigma,n}))$ of $X$, the pull-back of the functions $f$ and $g$ by $\hat \pi$ are given by
\begin{equation}\label{pomt1-exprpb}
\begin{aligned}
& \hat\pi^*f(\mathbf{y}_\sigma)=\tilde f_\sigma(\mathbf{y}_\sigma)\cdot\prod_{i=1}^n y_{\sigma,i}^{d(\mathbf{w}_i;f)}\quad\mbox{and}\\
& \hat\pi^*g(\mathbf{y}_\sigma)=\bigg(\tilde f_\sigma(\mathbf{y}_\sigma) + 
y_{\sigma,1}^{mw_1}\cdot\prod_{i=2}^n y_{\sigma,i}^{(d_1+m)w_{1,i}-d(\mathbf{w}_i;f)}\bigg)
\cdot\prod_{i=1}^n y_{\sigma,i}^{d(\mathbf{w}_i;f)}
\end{aligned}
\end{equation}
respectively, where
\begin{equation}\label{sect-sfm-expftg}
\begin{aligned}
\tilde f_\sigma(\mathbf{y}_\sigma) 
& =\sum_{\alpha}a_\alpha\, y_{\sigma,1}^{\ell_{\mathbf{w}_1}(\alpha)-d(\mathbf{w}_1;f)}\cdots\, y_{\sigma,n}^{\ell_{\mathbf{w}_n}(\alpha)-d(\mathbf{w}_n;f)}\\
& = \sum_{\alpha}a_\alpha\, y_{\sigma,2}^{\ell_{\mathbf{w}_2}(\alpha)-d(\mathbf{w}_2;f)}\cdots\, y_{\sigma,n}^{\ell_{\mathbf{w}_n}(\alpha)-d(\mathbf{w}_n;f)}
\end{aligned}
\end{equation}
(see \eqref{defeqforexcdiv1} and \eqref{defeqforexcdiv2}). Here, as in \eqref{linfct}, for each $1\leq i\leq n$ we write $\ell_{\mathbf{w}_i}(\alpha):=\sum_{j=1}^n w_{j,i}\alpha_j$ where as above ${}^t(w_{1,i},\ldots,w_{n,i})=\mathbf{w}_i$. The second equality in \eqref{sect-sfm-expftg} follows from the weighted homogeneity of $f$ with respect to the weight $\mathbf{w}_1=\mathbf{w}$, which implies that the difference $\ell_{\mathbf{w}_1}(\alpha)-d(\mathbf{w}_1;f)$ is zero for all indexes $\alpha$ that appear in the expression $f(\mathbf{z})=\sum_\alpha a_\alpha \mathbf{z}^\alpha$. So, in particular, $\tilde f_\sigma(\mathbf{y}_\sigma)$ does not depend on $y_{\sigma,1}$. Hereafter, we shall write $\tilde f'_\sigma(\mathbf{y}'_\sigma):=\tilde f_\sigma(\mathbf{y}_\sigma)$, where $\mathbf{y}'_\sigma:=(y_{\sigma,2},\ldots,y_{\sigma,n})$.
 
Now, let $\mathbf{p}$ be a singular point of $E(\mathbf{w})$ and let $(0,p_2,\ldots,p_n)$ be its coordinates in the chart $(\mathbb{C}^n_\sigma,\mathbf{y}_\sigma)$. Note that $p_i\not=0$ for any $2\leq i\leq n$. Indeed, for such $i$'s the point $\mathbf{p}$ cannot be in the intersection  $E(\mathbf{w}_i)\cap E(\mathbf{w})$ since the weighted homogeneity of $f$ implies $\dim \Delta(\mathbf{w}_i;f)\leq n-2$, and hence the Newton non-degeneracy assumption for $f^I$, $I\subsetneq \{1,\ldots,n\}$, implies that $f$ is Newton non-degenerate on $\Delta(\mathbf{w}_i;f)\cap\Delta(\mathbf{w}_1;f)$.
Consider the coordinates $(y_{\sigma,1},y_{\sigma,2}-p_2,\ldots,y_{\sigma,n}-p_n)$, which are centred at~$\mathbf{p}$.
Since $f$ is Newton pre-non-degenerate, there exists an analytic coordinate chart $(U_\mathbf{p},\mathbf{x}_\mathbf{p}=(x_{\mathbf{p},1},\ldots,x_{\mathbf{p},n}))$ of $X$ at $\mathbf{p}$ (i.e., $\mathbf{x}_\mathbf{p}(\mathbf{p})=\mathbf{0}$) such that $x_{\mathbf{p},1}=y_{\sigma,1}$, $\mathbf{x}'_\mathbf{p}:=(x_{\mathbf{p},2},\ldots,x_{\mathbf{p},n})$ is an analytic coordinate change of $(y_{\sigma,2}-p_2,\ldots,y_{\sigma,n}-p_n)$ \textemdash\ that is, there exists $\Phi=(\phi_2,\ldots,\phi_n)\in\mbox{Aut}(\mathbb{C}^{n-1})$ such that $y_{\sigma,i}-p_i=\phi_i(\mathbf{x}'_\mathbf{p})$ and $\phi_i(\mathbf{0})=0$ for any $2\leq i\leq n$ \textemdash\ and the defining function of the hypersurface $E(\mathbf{w})$ is convenient and Newton non-degenerate with respect to the coordinates $\mathbf{x}'_\mathbf{p}$ of $\hat E(\mathbf{w})$. Writing $\mathbf{p}':=(p_2,\ldots,p_n)$, we easily deduce from \eqref{pomt1-exprpb} that the pull-back of the functions $f$ and $g$ by $\hat \pi$ in the coordinates $\mathbf{x}_{\mathbf{p}}=(x_{\mathbf{p},1},\mathbf{x}'_\mathbf{p})$ are given by
\begin{equation}\label{rajout742022}
\begin{aligned}
& \hat\pi^*f(\mathbf{x}_{\mathbf{p}})= x_{\mathbf{p},1}^d \Psi(\mathbf{x}'_\mathbf{p}) 
\tilde f'_\sigma(\mathbf{p}'+\Phi(\mathbf{x}'_{\mathbf{p}}))
\quad\mbox{and}\\
& \hat\pi^*g(\mathbf{x}_{\mathbf{p}})=x_{\mathbf{p},1}^d \Psi(\mathbf{x}'_\mathbf{p}) \Big(\tilde f'_\sigma(\mathbf{p}'+\Phi(\mathbf{x}'_{\mathbf{p}}))+x_{\mathbf{p},1}^{mw_1}\Theta(\mathbf{x}'_{\mathbf{p}})\Big),
\end{aligned}
\end{equation}
where
\begin{align*}
\Psi(\mathbf{x}'_\mathbf{p}):=
\prod_{i=2}^n (p_i+\phi_i(\mathbf{x}'_\mathbf{p}))^{d(\mathbf{w}_i;f)}
\quad\mbox{and}\quad
\Theta(\mathbf{x}'_{\mathbf{p}}):=
\prod_{i=2}^n (p_i+\phi_i(\mathbf{x}'_\mathbf{p}))^{(d_1+m)w_{1,i}-d(\mathbf{w}_i;f)}.
\end{align*}
By the Newton pre-non-degeneracy of $f$, the defining function $\tilde f'_\sigma(\mathbf{p}'+\Phi(\mathbf{x}'_{\mathbf{p}}))$ of the hypersurface $E(\mathbf{w})$ is convenient and Newton non-degenerate with respect to the coordinates $\mathbf{x}'_\mathbf{p}$ of $\hat E(\mathbf{w})$, and since
\begin{equation*}
\Psi(\mathbf{0})=\prod_{i=2}^n p_i^{d(\mathbf{w}_i;f)}\not=0
\quad\mbox{and}\quad
\Theta(\mathbf{0})=\prod_{i=2}^n p_i^{(d_1+m)w_{1,i}-d(\mathbf{w}_i;f)}\not=0,
\end{equation*} 
it follows that in the coordinates $\mathbf{x}_\mathbf{p}$ the function $\hat\pi^*g$ is pseudo convenient, the Newton boundaries of $\hat\pi^*g(\mathbf{x}_\mathbf{p})$ and $x_{\mathbf{p},1}^d (\tilde f'_\sigma(\mathbf{p}'+\Phi(\mathbf{x}'_{\mathbf{p}}))+x_{\mathbf{p},1}^{mw_1})$ are the same, and $\hat\pi^*g$ is Newton non-degenerate. In particular, this shows that the function $g$ is almost Newton non-degenerate.

Since $\hat\pi^*g(\mathbf{x}_\mathbf{p})$ is pseudo convenient, there exists a subdivision $\Sigma^*_{\mathbf{p}}$ of $W_{\mathbb{R}}^+$ which is admissible with respect to the dual Newton diagram $\Gamma^*(\hat\pi^*g;\mathbf{x}_\mathbf{p})$ of $\hat\pi^*g$ with respect to the coordinates $\mathbf{x}_\mathbf{p}$ and such that all the vertices of $\Sigma^*_{\mathbf{p}}$ are positive except the $\mathbf{e}_i$'s ($1\leq i\leq n$). 
Let $\hat\omega_{\mathbf{p}}\colon Y_\mathbf{p}\to U_{\mathbf{p}}$ be the toric modification associated with $\Sigma^*_{\mathbf{p}}$, and let $\hat\omega\colon Y\to X$ be the canonical gluing of the union of these toric modifications as $\mathbf{p}$ runs over all the singular points of $E(\mathbf{w})$. Then the composition
\begin{equation*}
\hat\Pi\colon Y \xrightarrow{\ \hat\omega \ } X \xrightarrow{\ \hat\pi \ } \mathbb{C}^n
\end{equation*}
gives a good resolution of $g$ and the exceptional divisors of $\hat\Pi$ are all compact. In particular, this implies that $g$ has an isolated singularity at the origin, and its Milnor number $\mu_\mathbf{0}(g)$ can be computed from the zeta-function $\zeta_{g,\mathbf{0}}(t)$ of the monodromy associated with the Milnor fibration of $g$ at $\mathbf{0}$ using the formula \eqref{oka-relmilnornumber}.
Now, since in our case the set $P_0$ that appears in the formulas \eqref{okaformula11}--\eqref{okaformula13} reduces to the single (primitive positive) weight vector $\mathbf{w}$ \textemdash\ which is associated with the unique maximal dimensional face $\Delta(\mathbf{w};g)=\Delta(\mathbf{w};f)$ of $\Gamma(g)=\Gamma(f)$ \textemdash\ and since $d(\mathbf{w};g)=d(\mathbf{w};f)=d$, these formulas give
\begin{equation*}
\zeta_{g,\mathbf{0}}(t)=\zeta_{g_s}(t)\times (1-t^{d})^{(-1)^{n-1}\mu^{\mbox{\tiny tot}}} \times \prod_{\substack{\mathbf{p}\in\mbox{\tiny Sing}(E(\mathbf{w})})} \zeta_{\hat g,\mathbf{p}}(t),
\end{equation*}
so that the formula \eqref{oka-relmilnornumber} for the Milnor number is written as
\begin{equation*}%\label{mnog}
\mu_{\mathbf{0}}(g)=(-1)^n+(-1)^n\deg \zeta_{g_s}(t) + 
\sum_{\mathbf{p}\in\mbox{\tiny Sing}(E(\mathbf{w})}\big((-1)^n\deg \zeta_{\hat g,\mathbf{p}}(t)-d\mu_{\mathbf{p}}\big).
\end{equation*}
Here, $\mu_{\mathbf{p}}$ is the Milnor number of the hypersurface singularity $(E(\mathbf{w}),\mathbf{p})$, $\mu^{\mbox{\tiny tot}}$ is the sum (over all $\mathbf{p}\in\mbox{Sing}(E(\mathbf{w}))$) of the $\mu_{\mathbf{p}}$'s, and the family $\{g_s(\mathbf{z})\}_{|s|<1}$ is an analytic deformation of $g$ obtained from a small perturbation of the coefficients of the face function 
\begin{equation*}
g_{\Delta(\mathbf{w};g)}=f_{\Delta(\mathbf{w};f)}=f
\end{equation*}
 such that $g_s$ is Newton non-degenerate for all $s\not=0$. Again, we emphasize that the zeta-function $\zeta_{g_s,\mathbf{0}}(t)$ of the monodromy associated with the Milnor fibration  of $g_s$ at $\mathbf{0}$ is independent of $s\not=0$, and $\zeta_{g_s}(t)$ is nothing but a notation for the zeta-function $\zeta_{g_s,\mathbf{0}}(t)$ for $s\not=0$. Finally, in the above formula, in order to simplify and as in \S\ref{subsect-OF}, we have written $\hat g:=\hat\pi^*g$.

To establish the formula \eqref{edmt1}, it remains to compute $\deg \zeta_{g_s,\mathbf{0}}(t)$ for $s\not=0$ and $\deg \zeta_{\hat g,\mathbf{p}}(t)$. Since $g_s$, $s\not=0$, and $\hat g$ are Newton non-degenerate, we can apply the formula \eqref{var-relmilnornumber}. Pick any $s\not=0$, and let us start with the calculation of $\deg \zeta_{g_s,\mathbf{0}}(t)$. By \eqref{var-relmilnornumber}, we have
\begin{equation*}
\deg \zeta_{g_s,\mathbf{0}}(t)=-1+(-1)^n\mu_{\mathbf{0}}(g_s),
\end{equation*}
and we must compute $\mu_{\mathbf{0}}(g_s)$. As $g_s$ is convenient and Newton non-degenerate, the Milnor numbers at $\mathbf{0}$ of $g_s$ and of the face function $(g_s)_{\Delta(\mathbf{w};g)}$ are equal, and since $(g_s)_{\Delta(\mathbf{w};g)}$ is weighted homogeneous of weighted degree $d$ with respect to the weight $\mathbf{w}=(w_1,\ldots,w_n)$, the Milnor--Orlik formula \cite{MO} says that the Milnor number $\mu_{\mathbf{0}}((g_s)_{\Delta(\mathbf{w};g)})$ is given by
\begin{equation*}
\mu_{\mathbf{0}}((g_s)_{\Delta(\mathbf{w};g)})=\prod_{i=1}^n\bigg(\frac{d}{w_i}-1\bigg)=\prod_{i=1}^n(d_i-1).
\end{equation*}
So, altogether, we have 
\begin{equation}\label{part1-deg}
\deg \zeta_{g_s,\mathbf{0}}(t)=-1+(-1)^n\prod_{i=1}^n(d_i-1).
\end{equation}

Now let us compute $\deg \zeta_{\hat g,\mathbf{p}}(t)$. 
Since $\hat g\equiv \hat\pi^*g$ is Newton non-degenerate in the coordinates $\mathbf{x}_{\mathbf{p}}$, Varchenko's formula \eqref{varchenkoformula} shows that
\begin{align}\label{psf-dfzhdp}
\deg \zeta_{\hat g,\mathbf{p}}(t)
= \sum_{I\in\mathcal{I}}\sum_{\mathbf{v}\in P^I} 
-d(\mathbf{v};\hat g^I)\, \chi(\mathbf{v})
\end{align}
where $\mathcal{I}$ is the collection of all non-empty subsets $I\subseteq\{1,\ldots,n\}$ such that $\hat g^I\not\equiv 0$ (in particular, observe that since all the monomials of $\hat g$ contain a power of $x_{\mathbf{p},1}$, all the subsets $I\in \mathcal{I}$ contain the number~$1$) and $P^I$ is the set of primitive positive weight vectors in $W^{+I}$ which correspond to the maximal dimensional faces of $\Gamma(\hat g^I)$, that is, the set of vectors $\mathbf{v}={^t(v_1,\ldots,v_n)}\in W$ such that
\begin{equation*}
v_i>0\mbox{ for }i\in I,\ v_i=0\mbox{ for }i\notin I 
\mbox{ and }\dim \Delta(\mathbf{v};\hat g^I)=|I|-1.
\end{equation*}
Here, $\Delta(\mathbf{v};\hat g^I)$ is the face of $\Gamma(\hat g;\mathbf{x}_\mathbf{p})^I:=\Gamma(\hat g;\mathbf{x}_\mathbf{p})\cap \mathbb{R}^I=\Gamma(\hat g^I;\mathbf{x}_\mathbf{p}^I)$ associated to $\mathbf{v}$, where $\Gamma(\hat g;\mathbf{x}_\mathbf{p})$ is the Newton boundary of $\hat g$ with respect to the coordinates $\mathbf{x}_\mathbf{p}$ and where $\mathbf{x}_\mathbf{p}^I$ denote the coordinates on $\mathbb{C}^I$ induced by $\mathbf{x}_\mathbf{p}$.
We recall that the number $\chi(\mathbf{v})$ is defined by 
\begin{equation*}
\chi(\mathbf{v}):=(-1)^{|I|-1}\, |I|!\, \mbox{Vol}_{|I|}\big(\mbox{Cone}(\Delta(\mathbf{v};\hat g^I),\mathbf{0}^I)\big) / d(\mathbf{v};\hat g^I).
\end{equation*} 
Writing down \eqref{psf-dfzhdp} explicitly gives
\begin{align*}
\deg \zeta_{\hat g,\mathbf{p}}(t)
& = \sum_{I\in\mathcal{I}} -(-1)^{|I|-1}\, |I|! \sum_{\mathbf{v}\in P^I}  
\mbox{Vol}_{|I|}\big(\mbox{Cone}(\Delta(\mathbf{v};\hat g^I),\mathbf{0}^I)\big)\\
& = \sum_{I\in\mathcal{I}} -(-1)^{|I|-1}\, |I|! \,  
\mbox{Vol}_{|I|}\big(\Gamma_{\!-}(\hat g)^I\big)
\end{align*}
where $\Gamma_{\!-}(\hat g)$ is the cone over $\Gamma(\hat g):=\Gamma(\hat g;\mathbf{x}_\mathbf{p})$ with the origin as vertex and $\Gamma_{\!-}(\hat g)^I:=\Gamma_{\!-}(\hat g)\cap \mathbb{R}^I$.
Now, by \eqref{rajout742022} and \cite[Assertion 19]{O2}, for any subset $I\subseteq \mathcal{I}$, we have
\begin{equation*}
|I|!\, \mbox{Vol}_{|I|}(\Gamma_{\!-}(\hat g)^I)=
(d+mw_1)\, |I'|!\ \mbox{Vol}_{|I'|}
\big(\Gamma_{\!-}(\tilde{f}'_\sigma(\mathbf{p}'+\Phi(\mathbf{x}'_\mathbf{p})))^{I'}\big),
\end{equation*}
where $I':=I\setminus \{1\}$.
(We recall that for any $I\in\mathcal{I}$, we have $1\in I$. If $I=\{1\}$, then the right-hand side of the above equality is $1$ by definition.)
Writing $\Gamma_{\!-}(\tilde{f}'_\sigma)$ instead of $\Gamma_{\!-}(\tilde{f}'_\sigma(\mathbf{p}'+\Phi(\mathbf{x}'_\mathbf{p})))$, it follows that
\begin{align*}
\deg \zeta_{\hat g,\mathbf{p}}(t)
& = (d+mw_1) \sum_{I\in\mathcal{I}} -(-1)^{|I|-1}\, |I'|!\ \mbox{Vol}_{|I'|} \big(\Gamma_{\!-}(\tilde{f}'_\sigma)^{I'}\big)\\
& = (-1)^n\, (d+mw_1) \sum_{I\in\mathcal{I}} (-1)^{(n-1)-|I'|}\, |I'|!\ \mbox{Vol}_{|I'|} \big(\Gamma_{\!-}(\tilde{f}'_\sigma)^{I'}\big)
\end{align*}
where the sum $\sum_{I\in\mathcal{I}}(\ldots)$ is (by definition) the Newton number of the convenient function $\tilde{f}'_\sigma(\mathbf{p}'+\Phi(\mathbf{x}'_\mathbf{p}))$ with respect to the coordinates $\mathbf{x}'_\mathbf{p}$. Now, since this function is Newton non-degenerate, a theorem of Kouchnirenko \cite[Th\'eor\`eme 1.10]{K} says that its Newton number coincides with its Milnor number at $\mathbf{p}'$ (which is nothing but the Milnor number $\mu_\mathbf{p}$ of the hypersurface singularity $(E(\mathbf{w}),\mathbf{p})$). Thus,
\begin{align*}
\deg \zeta_{\hat g,\mathbf{p}}(t) = (-1)^n (d+mw_1) \, \mu_\mathbf{p}.
\end{align*}

Altogether, we get that the Milnor number $\mu_{\mathbf{0}}(g)$ of $g$ at $\mathbf{0}$ is equal to
\begin{align*}
(-1)^n+(-1)^n  \bigg(-1+(-1)^n\prod_{i=1}^n(d_i-1)\bigg) + 
 \sum_{\mathbf{p}\in\mbox{\tiny Sing}(E(\mathbf{w}))}\big((-1)^n((-1)^n (d+mw_1) \, \mu_\mathbf{p})-d\mu_{\mathbf{p}}\big),
\end{align*}
that is, 
\begin{align*}
\mu_{\mathbf{0}}(g)=\prod_{i=1}^n (d_i-1) + 
m\, w_1\, \mu^{\mbox{\tiny tot}}.
\end{align*}
This completes the proof of Theorem \ref{mt1}.
\end{proof}

\begin{example}\label{example33}
Take $n=3$ and suppose that 
\begin{align*}
f(z_1,z_2,z_3)=7z_3^6+5z_1z_3^4+12z_2z_3^4-8z_1^2z_3^2+6z_2^2z_3^2+4z_1^3+z_2^3.
\end{align*}
Clearly, $f$ is a convenient weighted homogeneous polynomial function of weighted degree $d=6$ with respect to the weight vector $\mathbf{w}={}^t(2,2,1)$. We have $f(\mathbf{0})=0$ and the singular locus of $V=f^{-1}(0)$ is $1$-dimensional. Also, we easily check that for any proper subset $I\subsetneq\{1,\ldots,n\}$, the function $f^I$ is Newton non-degenerate. The integers $d_i$  that appear in Theorem \ref{mt1} are given by $d_1=3$, $d_2=3$ and $d_3=6$. The dual Newton diagram $\Gamma^*(f)$ has a single positive vertex, namely $\mathbf{w}={}^t(2,2,1)$. Consider the regular simplicial cone subdivision $\Sigma^*$ of $W^+_{\mathbb{R}}$ whose vertices are $\mathbf{w}={}^t(2,2,1)$, $\mathbf{v}={}^t(1,1,1)$ and the canonical weight vectors $\mathbf{e}_i$ ($1\leq i\leq 3$). Clearly, it is admissible with respect to $\Gamma^*(f)$. Let $\hat\pi\colon X\to\mathbb{C}^3$ be the associated toric modification. It has five toric coordinate charts, which correspond to the cones $\sigma:=C(\mathbf{w},\mathbf{v},\mathbf{e}_1)$, $\sigma':=C(\mathbf{w},\mathbf{v},\mathbf{e}_2)$, $\sigma'':=C(\mathbf{w},\mathbf{e}_1,\mathbf{e}_2)$, $\sigma''':=C(\mathbf{v},\mathbf{e}_1,\mathbf{e}_3)$ and $\sigma'''':=C(\mathbf{v},\mathbf{e}_2,\mathbf{e}_3)$ (see Figure \ref{figure}). 
\begin{figure}[t]
\includegraphics[scale=2]{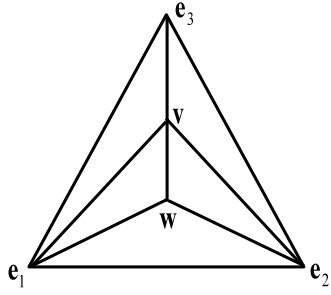}
\caption{The subdivision $\Sigma^*$ of Example \ref{example33}}
\label{figure}
\end{figure}
In the chart $\mathbb{C}^3_\sigma$, with coordinates $\mathbf{y}_\sigma=(y_{\sigma,1},y_{\sigma,2},y_{\sigma,3})$, we have the birational map
\begin{equation*}
\hat\pi_\sigma\colon\mathbb{C}^3_\sigma\to\mathbb{C}^3,\ (\mathbf{y}_\sigma)\mapsto (y_{\sigma,1}^2y_{\sigma,2}y_{\sigma,3},y_{\sigma,1}^2y_{\sigma,2},y_{\sigma,1}y_{\sigma,2})
\end{equation*}
(see \eqref{def-bmpsh}). Now consider, for instance, the function 
\begin{equation*}
g_2(\mathbf{z})=f(\mathbf{z})+z_2^{d_2+m}=f(\mathbf{z})+z_2^{3+m} \quad (m\geq 1).
\end{equation*}
The pull-back of $f$ and $g_2$ by $\hat \pi$ are given by
\begin{align*}
& \hat f(\mathbf{y}_\sigma):=\hat\pi^*f(\mathbf{y}_\sigma)=y_{\sigma,1}^6 y_{\sigma,2}^3\cdot \tilde f_\sigma(\mathbf{y}_\sigma)\quad\mbox{and}\\
& \hat g_2(\mathbf{y}_\sigma):=\hat\pi^*g_2(\mathbf{y}_\sigma)=y_{\sigma,1}^6 y_{\sigma,2}^3 \cdot (\tilde f_\sigma(\mathbf{y}_\sigma)+y_{\sigma,1}^{2m} y_{\sigma,2}^m)
\end{align*}
respectively, where
\begin{align}\label{ex-definingequation}
\tilde f_\sigma(\mathbf{y}_\sigma) = \tilde f'_\sigma(\mathbf{y}'_\sigma) =
7y_{\sigma,2}^3+5y_{\sigma,2}^2y_{\sigma,3}+12y_{\sigma,2}^2-8y_{\sigma,2}y_{\sigma,3}^2+6y_{\sigma,2}+4y_{\sigma,3}^3+1
\end{align}
(see \eqref{defeqforexcdiv1} and \eqref{defeqforexcdiv2}).
The exceptional divisor $E(\mathbf{w})$ has a unique singularity at $\mathbf{p}=(0,-1/2,-1/4)$. It is a singularity of type $\mathbf{A}_2$, so that its Milnor number $\mu_{\mathbf{p}}$ equals $2$. In the admissible coordinates $\mathbf{x}_{\mathbf{p}}=(x_{\mathbf{p},1},x_{\mathbf{p},2},x_{\mathbf{p},3})$ defined by 
\begin{align*}
& y_{\sigma,1}=x_{\mathbf{p},1},\\ 
& y_{\sigma,2}=(x_{\mathbf{p},2}-(1/2))+2x_{\mathbf{p},3},\\
& y_{\sigma,3}=x_{\mathbf{p},3}-(1/4),
\end{align*}
the Newton principal part of the defining polynomial \eqref{ex-definingequation} of the hypersurface $E(\mathbf{w})$ of $\hat E(\mathbf{w})$ is given by
\begin{align*}
(1/4) x_{\mathbf{p},2}^2+64 x_{\mathbf{p},3}^3,
\end{align*}
which is clearly convenient and Newton non-degenerate. So, all the conditions for applying Theorem \ref{mt1} are fulfilled, and we get
\begin{align}\label{rajout842022}
\mu_{\mathbf{0}}(g_2)=20+4m.
\end{align}
\end{example}

\begin{remark}
We can check the expression \eqref{rajout842022} of the Milnor number by computing the degree of the zeta-function $\zeta_{g_2,\mathbf{0}}(t)$. The latter is given by the formulas \eqref{okaformula11}--\eqref{okaformula13}. More precisely, the zeta-function $\zeta_{\hat g_2,\mathbf{0}}(t)$ that appears in \eqref{okaformula12} \textemdash\ and that corresponds in our case to the singular point $\mathbf{p}=(0,-1/2,-1/4)$ \textemdash\ can be calculated using Varchenko's formula \eqref{varchenkoformula}. Explicitly, the Newton principal part of $\hat g_2$ is written as
\begin{equation*}
(-1/2)^3 x_{\mathbf{p},1}^6 \big((1/4)x_{\mathbf{p},2}^2+64x_{\mathbf{p},3}^3+(-1/2)^mx_{\mathbf{p},1}^{2m}\big),
\end{equation*}
and $\zeta_{\hat g_2,\mathbf{0}}(t)$ is given by
\begin{equation*}
\zeta_{\hat g_2,\mathbf{0}}(t)=\left\{
\begin{aligned}
& (1-t^{6m+18})^{-1}(1-t^{2m+6}) && \mbox{if}  && \gcd(m,3)=1,\\
& (1-t^{2m+6})^{-2} && \mbox{if}  && \gcd(m,3)=3.
\end{aligned}
\right.
\end{equation*}
The zeta-function $\zeta_{(g_2)_s,\mathbf{0}}(t)$ that appears in \eqref{okaformula13} is also computed using Varchenko's formula and is given by
\begin{equation*}
\zeta_{(g_2)_s,\mathbf{0}}(t)=(1-t^3)(1-t^6)^{-4},
\end{equation*}
so that the zeta-function $\zeta'(t)$ of \eqref{okaformula13} is written as
\begin{equation*}
\zeta'(t)=\zeta_{(g_2)_s,\mathbf{0}}(t)\cdot(1-t^6)^2=(1-t^3)(1-t^6)^{-2}.
\end{equation*}
Thus, altogether, the zeta-function $\zeta_{g_2,\mathbf{0}}(t)$ (given by \eqref{okaformula12}) is written as
\begin{equation*}
\zeta_{g_2,\mathbf{0}}(t)=\left\{
\begin{aligned}
& (1-t^{6m+18})^{-1}(1-t^{2m+6})(1-t^3)(1-t^6)^{-2} && \mbox{if}  && \gcd(m,3)=1,\\
& (1-t^{2m+6})^{-2}(1-t^3)(1-t^6)^{-2} && \mbox{if}  && \gcd(m,3)=3.
\end{aligned}
\right.
\end{equation*}
Though the expression for zeta-function differs according to the cases $\gcd(m,3)=1$ or $\gcd(m,3)=3$, its degree is the same in both cases, and therefore, by \eqref{lprrmnedzf}, we get 
\begin{equation*}
\mu_{\mathbf{0}}(g_2)=-\deg \zeta_{g_2,\mathbf{0}}(t) + 1 = 20+4m,
\end{equation*}
which is the assertion of Theorem \ref{mt1} in the situation of Example \ref{example33}.
\end{remark}

%%%%%%%%%%%%%%%%%%%%%%%%%%%%%%%%%%%%%%%%%%%%%%%%%%%%%%%%%%%%%%%%%%%%%%%%%%%%%%%%%%%%%%%%%%%%%
\section{On the structure of the $\mu$-constant and $\mu^*$-constant strata}\label{sect-smcmscs}
%%%%%%%%%%%%%%%%%%%%%%%%%%%%%%%%%%%%%%%%%%%%%%%%%%%%%%%%%%%%%%%%%%%%%%%%%%%%%%%%%%%%%%%%%%%%%
As mentioned in the introduction, in order to construct our $\mu^*$-Zariski pair of surfaces, we need to show that if $f$ and $f'$ are two polynomial functions vanishing at the origin and lying in the same path-connected component  of the $\mu^*$-constant stratum (as germs of analytic functions at the origin), then $f$ and $f'$  can be connected by a ``piecewise complex-analytic path'' (see Definition \ref{def-pcap} and the comment after it).
The purpose of this section is to establish this property. We shall also prove a similar property for the $\mu$-constant stratum. The main result of the section is stated in Theorem \ref{mt3}. Certainly, this theorem may be useful in many other situations in singularity theory.

The definitions of piecewise complex-analytic paths in the $\mu$-constant and $\mu^*$-constant strata are based on the properties of certain semi-algebraic sets $W(n,m,\mu)$ and $W^*(n,m,\mu^*)$ that we are going to introduce in \S\S\ref{charact-mcs} and \ref{charact-mscs}. The definition of piecewise complex-analytic paths itself and the statement of Theorem \ref{mt3} are given in \S\ref{subsect-pccapcap}.

Let $\mathcal{O}_n\equiv \mathbb{C}\{z_1,\ldots,z_n\}$ ($n\geq 1$) be the ring of convergent power series at the origin, and let $\mathfrak{M}:=\{f\in \mathcal{O}_n\mid f(\mathbf{0})=0\}$ be its maximal ideal. It is well known that for a given $f\in\mathfrak{M}$, if $H$ is a generic linear $i$-plane of $\mathbb{C}^n$ ($1\leq i\leq n$), then the Milnor number 
\begin{equation*}
\mu_{\mathbf{0}}^{(i)}(f):=\mu_{\mathbf{0}}(f\vert_H)
\end{equation*}
 of the restriction of $f$ to $H$ depends only on $i$ and $f$. 
(Note that for a non-generic linear $i$-plane $L$, we have $\mu_{\mathbf{0}}^{(i)}(f)\leq\mu_{\mathbf{0}}(f\vert_L)$.)
In \cite{Teissier2}, Teissier introduced the $\mu^*$-sequence of $f$ at~$\mathbf{0}$ as the $n$-tuple
\begin{equation*}
\mu^*_{\mathbf{0}}(f):=(\mu_{\mathbf{0}}^{(n)}(f),\mu_{\mathbf{0}}^{(n-1)}(f),\ldots,\mu_{\mathbf{0}}^{(1)}(f)).
\end{equation*}
Note that $\mu_{\mathbf{0}}^{(n)}(f)$  is nothing but the Milnor number $\mu_{\mathbf{0}}(f)$ of $f$ at $\mathbf{0}$ while $\mu_{\mathbf{0}}^{(1)}(f)$ is the multiplicity $\mbox{mult}_{\mathbf{0}}(f)$ of $f$ at $\mathbf{0}$ minus 1.

By definition, if $\mu$ is a non-negative integer, then the $\mu$-constant stratum $\mathfrak{M}(\mu)$ of $\mathfrak{M}$ consists of all function-germs $f\in \mathfrak{M}$ such that the Milnor number $\mu_{\mathbf{0}}(f)$ of $f$ at $\mathbf{0}$ is equal to $\mu$.
Similarly, if $\mu^{(n)},\ldots,\mu^{(1)}$ are non-negative integers and if $\mu^*$ denotes the $n$-tuple $(\mu^{(n)},\mu^{(n-1)},\ldots,\mu^{(1)})$, then the $\mu^*$-constant stratum $\mathfrak{M}(\mu^*)$ of $\mathfrak{M}$ consists of all function-germs $f\in \mathfrak{M}$ such that the $\mu^*$-sequence $\mu^*_{\mathbf{0}}(f)$ of $f$ at $\mathbf{0}$ is given by the $n$-tuple $\mu^*$.

%%%%%%%%%%%%%%%%%%%%%%%%%%%%%%%%%%%%%%%%%%%%%%%%%%%%%%%%%%%%%%%%%%%%%%%%%%%%%%%%%%%%%%%%%%%%
\subsection{The semi-algebraic set $W(n,m,\mu)$}\label{charact-mcs}
%%%%%%%%%%%%%%%%%%%%%%%%%%%%%%%%%%%%%%%%%%%%%%%%%%%%%%%%%%%%%%%%%%%%%%%%%%%%%%%%%%%%%%%%%%%%
Now, let $m$ be a positive integer and let 
\begin{align*}
P(n,m):=\{f\in\mathbb{C}[z_1,\ldots,z_n]\mid \deg f\leq m\}.
\end{align*} 
It is well known that $P(n,m)$ is a vector space of dimension $N:=\binom{n+m}{n}$, putting an order on the basis' monomials $1=M_1<M_2<\cdots<M_N$. Hereafter, we identify $P(n,m)$ with $\mathbb{C}^N$.
Let 
\begin{equation*}
\pi_m\colon \mathcal{O}_n\to P(n,m)
\end{equation*}
be the natural projection obtained by deleting all terms of degree greater than $m$ (if any). For any $f\in\mathcal{O}_n$, we write $J(f)$ for the Jacobian ideal of $f$ (i.e., the ideal of $\mathcal{O}_n$ generated by the partial derivatives of $f$) and we put 
\begin{equation*}
J_m(f):=\pi_m(J(f)).
\end{equation*} 
 An  element of $J_m(f)$ is written as $\pi_m(\sum_{i=1}^n h_i\, (\partial f/\partial z_i))$, $h_i\in\mathcal{O}_n$, and we easily check that $J_m(f)$ is the subspace of $P(n,m)\equiv\mathbb{C}^N$ generated by the following set of $nN$ vectors:
\begin{equation*}
B(f):=\bigg\{\pi_m\bigg( M_j\frac{\partial f}{\partial z_i}\bigg) \mid
1\leq i\leq n,\ 1\leq j\leq N\bigg\}.
\end{equation*} 
Hereafter, we identify $B(f)$ with an $N\times nN$ matrix. 
Then for any $\mu\in\mathbb{N}$, we consider the algebraic variety\footnote{By identifying $f=\sum_{j=1}^N a_j M_j$ with its coordinates $(a_1,\ldots,a_N)\in\mathbb{C}^N$ with respect to the basis $\{M_1,\ldots,M_N\}$, we immediately see that $V(n,m,\mu)$ is an algebraic variety.}
\begin{align*}
V(n,m,\mu):=\{f \in P(n,m) \mid \mbox{all $(N-\mu)\times(N-\mu)$ 
minors of $B(f)$ vanish}\},
\end{align*}
and we define
\begin{equation}\label{defsasmcs}
W(n,m,\mu):=V(n,m,\mu-1)\setminus V(n,m,\mu).
\end{equation}  
Clearly, $W(n,m,\mu)$ is a semi-algebraic set, and for any $f\in P(n,m)$, the following equivalences hold true:
\begin{equation}\label{equivalences}
f\in W(n,m,\mu)\Leftrightarrow \mbox{rk}\, B(f)=N-\mu \Leftrightarrow \dim P(m,n)/J_m(f)=\mu.
\end{equation} 

The next proposition is a crucial step in the proof of Theorem \ref{mt3}, the main result of this section. To state it, let us consider the set
\begin{equation*}
W(n,\mu):=\{f\in\mathcal{O}_n\mid \dim \mathcal{O}_n/J(f)=\mu\}.
\end{equation*}

\begin{remark}
The $\mu$-constant stratum $\mathfrak{M}(\mu)$ of $\mathfrak{M}$ is nothing but $\mathfrak{M}\cap W(n,\mu)$.
\end{remark}

\begin{proposition}\label{mucs-lemmaequiv} 
Let $f\in\mathfrak{M}$ (i.e., $f(\mathbf{0})=0$).
\begin{enumerate}
\item
If $f\in W(n,\mu)$, then $\pi_m(f)\in W(n,m,\mu)$ for any $m\geq\mu$.
\item
If $f\in W(n,m,\mu)$ for some $m\geq\mu$, then $f\in W(n,\mu)$.
\end{enumerate}
\end{proposition}

\begin{proof}
Item (1) is easy. Take $f\in W(n,\mu)$. Then $\dim \mathcal{O}_n/J(f)=\mu$, and it is well known that this implies $\mathfrak{M}^{\mu}\subseteq J(f)$.
Thus for any $m\geq \mu$, we have $\mathfrak{M}^{m+1}\subseteq\mathfrak{M}^{\mu}\subseteq J(f)$, and hence,
\begin{equation*}
\mathcal{O}_n/J(f)=\mathcal{O}_n/(J(f)+\mathfrak{M}^{m+1})=P(n,m)/J_m(f).
\end{equation*} 
It follows that $\dim P(n,m)/J_m(f)=\mu$, and we conclude with \eqref{equivalences}.

Let us now prove item (2). Writing $\bar{\mathfrak{M}}$ for the canonical image of $\mathfrak{M}$ in 
\begin{equation*}
\mathcal{O}_n/(J(f)+\mathfrak{M}^{m+1})=P(n,m)/J_m(f),
\end{equation*} 
we look at the canonical decomposition
\begin{equation*}
P(n,m)/J_m(f)=\mathcal{O}_n/(J(f)+\mathfrak{M}^{m+1})=\bigoplus_{r=0}^m \bar{\mathfrak{M}}^r/\bar{\mathfrak{M}}^{r+1}.
\end{equation*} 
Clearly, there exists $0\leq r_0\leq\mu$ such that $\bar{\mathfrak{M}}^{r_0}/\bar{\mathfrak{M}}^{r_0+1}=0$, as otherwise $\dim P(n,m)/J_m(f)>\mu$, which is a contradiction. In particular, this implies
\begin{equation*}
\mathfrak{M}^{r_0}\subseteq (J(f)+\mathfrak{M}^{m+1})+\mathfrak{M}^{r_0+1}.
\end{equation*}
Now, since $0\leq r_0\leq \mu\leq m$, we also have $\mathfrak{M}^{m+1}\subseteq \mathfrak{M}^{r_0+1}$, a new inclusion which, combined with the above one, shows that
\begin{equation}\label{nai}
{\mathfrak{M}}^{r_0}\subseteq \mathfrak{M}^{r_0+1}+J(f).
\end{equation}
Clearly, \eqref{nai} implies that for any $k\geq r_0$ the following equality holds:
\begin{equation}\label{le}
{\mathfrak{M}}^{k}+J(f)={\mathfrak{M}}^{r_0}+J(f).
\end{equation}
This equality, in turn, shows that $f$ has an isolated singularity at $\mathbf{0}$
(i.e., there exists $\mu'>0$ such that $f\in W(n,\mu')$). Indeed, if not, then
\begin{equation}\label{efc}
\dim \mathcal{O}_n/(J(f)+\mathfrak{M}^{k})\to\infty \quad\mbox{as } k\to\infty.
\end{equation}
However, by \eqref{le}, for any $k\geq r_0$, we have
\begin{equation*}
\mathcal{O}_n/(J(f)+\mathfrak{M}^{k})\overset{\mbox{\tiny \eqref{le}}}{=}\mathcal{O}_n/(J(f)+\mathfrak{M}^{r_0})\overset{\mbox{\tiny \eqref{le}}}{=}\mathcal{O}_n/(J(f)+\mathfrak{M}^{m+1})=P(n,m)/J_m(f),
\end{equation*}
and therefore
\begin{equation*}
\dim \mathcal{O}_n/(J(f)+\mathfrak{M}^{k})=\dim P(n,m)/J_m(f)=\mu,
\end{equation*}
which contradicts \eqref{efc}.
Now, since $f$ has an isolated singularity at $\mathbf{0}$, it follows from the Hilbert Nullstelensatz (see, e.g., \cite[Chapter VII, \S 3, Theorem 14]{ZS}) that there exists $\ell>0$ such that $\mathfrak{M}^\ell\subseteq J(f)$. Clearly, we can assume $\ell\geq m$. Then,
\begin{align*}
\mu'=\dim \mathcal{O}_n/J(f) & = \dim \mathcal{O}_n/(\mathfrak{M}^{\ell}+J(f))\\
& \overset{\mbox{\tiny \eqref{le}}}{=} \dim \mathcal{O}_n/(\mathfrak{M}^{m+1}+J(f))
=\dim P(n,m)/J_m(f)=\mu.
\end{align*} 
In other words, $f\in W(n,\mu)$.
\end{proof}

The following corollary is an immediate consequence of Proposition \ref{mucs-lemmaequiv}.

\begin{corollary}
For any $m'\ge m\ge \mu$, the following inclusions are homotopy equivalences:
\begin{align*}
W(n,m,\mu)\cap\mathfrak{M}\hookrightarrow W(n,m',\mu)\cap\mathfrak{M}\hookrightarrow W(n,\mu)\cap\mathfrak{M}.
\end{align*} 
\end{corollary}

%%%%%%%%%%%%%%%%%%%%%%%%%%%%%%%%%%%%%%%%%%%%%%%%%%%%%%%%%%%%%%%%%%%%%%%%%%%%%%%%%%%%%%%%%%%%
\subsection{The semi-algebraic set $W^*(n,m,\mu^*_n)$}\label{charact-mscs}
%%%%%%%%%%%%%%%%%%%%%%%%%%%%%%%%%%%%%%%%%%%%%%%%%%%%%%%%%%%%%%%%%%%%%%%%%%%%%%%%%%%%%%%%%%%%

Let $\mu^{(n)},\ldots,\mu^{(1)}$ be non-negative integers, and let $\mu^*_{n}:=(\mu^{(n)},\mu^{(n-1)},\ldots,\mu^{(1)})$.  Hereafter, we are going to define a semi-algebraic set $W^*(n,m,\mu^*_n)\subseteq P(n,m)$ by induction on $n$. For that purpose, we consider the natural projection 
\begin{align*}
\mbox{pr}_1\colon P(n,m) \times\mathbb{C}^{n-1} \to P(n,m)
\end{align*} 
 onto the first factor and we introduce the map
\begin{align*}
\phi_n\colon P(n,m) \times\mathbb{C}^{n-1} \to P(n-1,m)
\end{align*} 
which associates to any $(f,b_1,\ldots,b_{n-1})\in P(n,m) \times\mathbb{C}^{n-1}$ the polynomial function defined by
\begin{align*}
(z_1,\ldots, z_{n-1})\mapsto f(z_1,\ldots, z_{n-1},b_1 z_1+\cdots+b_{n-1} z_{n-1}).
\end{align*} 

The induction starts at $n=1$ in which case we set
\begin{equation}\label{defsasmcs-2}
W^*(1,m,\mu^*_{1}):=W(1,m,\mu^{(1)}),
\end{equation} 
where $W(1,m,\mu^{(1)})$ is the semi-algebraic set defined in \eqref{defsasmcs}.
Now, suppose that for any $n\geq 2$ we have defined a semi-algebraic subset $W^*(n-1,m,\mu^*_{n-1})\subseteq P(n-1,m)$, and let us define a new semi-algebraic subset $W^*(n,m,\mu^*_{n})\subseteq P(n,m)$ by the relation
\begin{equation}\label{defsasmcs-3}
W^*(n,m,\mu^*_{n}):=A(n,m,\mu^*_{n})\setminus B(n,m,\mu^*_{n}), 
\end{equation} 
where
\begin{align*}
& A(n,m,\mu^*_{n}):=\mbox{pr}_1\Big( \phi_n^{-1}\big(  W^*(n-1,m,\mu^*_{n-1}) \big) \cap \big( W(n,m,\mu^{(n)})\times \mathbb{C}^{n-1} \big) \Big),\\
& B(n,m,\mu^*_{n}):=\mbox{pr}_1\bigg( \phi_n^{-1}\bigg( \bigcup_{s<\mu^{(n-1)}} W(n-1,m,s) \bigg) \bigg).
\end{align*} 
Again, $W(n,m,\mu^{(n)})$ and $W(n-1,m,s)$ are the semi-algebraic sets defined in \eqref{defsasmcs}. Note that $f\in A(n,m,\mu^*_{n})$ means $f\in W(n,m,\mu^{(n)})$ and there exists $(b_1,\ldots,b_{n-1})\in \mathbb{C}^{n-1}$ such that if 
\begin{align*}
H:=\{\mathbf{z}=(z_1,\ldots,z_n)\in\mathbb{C}^{n}\mid z_n=b_1 z_1+\cdots+b_{n-1} z_{n-1}\}
\end{align*} 
 denotes the corresponding hyperplane, then
\begin{align*}
\phi_n(f,b_1,\ldots,b_{n-1})=f\vert_H\in W^*(n-1,m,\mu^*_{n-1}).
\end{align*} 
 Saying $f\notin B(n,m,\mu^*_{n})$ means that the above hyperplane $H$ is generic. That $W^*(n,m,\mu^*_{n})$ is a semi-algebraic set follows from the Tarski--Seidenberg theorem (see, e.g., \cite{Coste}).

The proposition below is a consequence of Proposition \ref{mucs-lemmaequiv}. It also plays a crucial role in the proof of Theorem \ref{mt3}. Let 
\begin{equation*}
W^*(n,\mu^*_{n}):=\{f\in\mathcal{O}_n\mid \mu^*_{\mathbf{0}}(f-f(\mathbf{0}))=\mu^*_{n}\}.
\end{equation*}
Note that for $n=1$, we have $W^*(1,\mu^*_{1})=W(1,\mu^{(1)})$.

\begin{remark}
The $\mu^*$-constant stratum $\mathfrak{M}(\mu^*)$ of $\mathfrak{M}$ is nothing but $\mathfrak{M}\cap W^*(n,\mu^*_{n})$.
\end{remark}

\begin{proposition}\label{mucs-lemmaequiv-2}
Put $\mu^{\mbox{\tiny \emph{max}}}:=\mbox{\emph{max}}\, \{\mu^{(n)},\ldots,\mu^{(1)}\}$ and pick $f\in\mathfrak{M}$ (i.e., $f(\mathbf{0})=0$). 
\begin{enumerate}
\item
If $f\in W^*(n,\mu^*_n)$, then $\pi_m(f)\in W^*(n,m,\mu^*)$ for any 
$m\geq\mu^{\mbox{\tiny \emph{max}}}$.
\item
If $f\in W^*(n,m,\mu^*_n)$ for some $m\geq\mu^{\mbox{\tiny \emph{max}}}$, then $f\in W^*(n,\mu^*_n)$.
\end{enumerate}
\end{proposition}

\begin{proof}
Let us first show item (1). We argue by induction on $n$. By Proposition \ref{mucs-lemmaequiv}, if $f\in W^*(1,\mu^*_1)=W(1,\mu^{(1)})$, then $\pi_m(f)\in W(1,m,\mu^{(1)})=:W^*(1,m,\mu^*_1)$. For the inductive step, we assume that the following implication holds true:
\begin{equation*}
f\in W^*(n-1,\mu^*_{n-1})\Rightarrow \pi_m(f)\in W^*(n-1,m,\mu^*_{n-1})
\mbox{ for any } m\geq\mbox{max}\, \{\mu^{(n-1)},\ldots,\mu^{(1)}\}.
\end{equation*} 
Now take $f\in W^*(n,\mu^*_n)$. We want to show that $\pi_m(f)\in W(n,m,\mu^{(n)})$, and for $H$ generic, $\pi_m(f)\vert_H\in W^*(n-1,m,\mu^*_{n-1})$. We have:
\begin{equation*}
f\in W^*(n,\mu^*_n)\Rightarrow f\in W(n,\mu^{(n)})
\overset{\mbox{\tiny Prop.~\ref{mucs-lemmaequiv}}}{\Rightarrow}
\pi_m(f)\in W(n,m,\mu^{(n)}).
\end{equation*} 
Also,  for any $m\geq\mu^{\mbox{\tiny max}}$ we have:
\begin{align*}
f\in W^*(n,\mu^*_n) & \Rightarrow \mbox{ for any $H$ generic, } 
f\vert_H\in W^*(n-1,\mu^*_{n-1})\\
& \overset{\mbox{\tiny induction}}{\Rightarrow}
\pi_m(f\vert_H)=\pi_m(f)\vert_H\in W^*(n-1,m,\mu^*_{n-1}).
\end{align*} 

Let us now prove item (2). Again, we argue by induction on $n$. By Proposition \ref{mucs-lemmaequiv}, if $f\in W^*(1,m,\mu^*_1):=W(1,m,\mu^{(1)})$ for some $m\geq\mu^{(1)}$, then $f\in W(1,\mu^{(1)})=W^*(1,\mu^*_1)$. For the inductive step, we assume that the following implication holds true:
\begin{equation*}
f\in W^*(n-1,m,\mu^*_{n-1}) \mbox{ with } m\geq\mbox{max}\, \{\mu^{(n-1)},\ldots,\mu^{(1)}\}\Rightarrow  f\in W^*(n-1,\mu^*_{n-1}).
\end{equation*} 
Now take $f\in W^*(n,m,\mu^*_n)$ with $m\geq\mu^{\mbox{\tiny max}}$. Then, by definition, $f\in W(n,m,\mu^{(n)})$, and for $H$ generic, $f\vert_H\in W^*(n-1,m,\mu^*_{n-1})$. Thus, by the induction hypothesis, $f\vert_H\in W^*(n-1,\mu^*_{n-1})$. Altogether, $f\in W^*(n,\mu^*_{n})$.
\end{proof}

As an immediate corollary of Proposition \ref{mucs-lemmaequiv-2} we have the following statement.

\begin{corollary}
For any $m'\ge m\ge \mu$, the following inclusions are homotopy equivalences:
\begin{align*}
W^*(n,m,\mu^*_n)\cap\mathfrak{M}\hookrightarrow W^*(n,m',\mu^*_n)\cap\mathfrak{M}\hookrightarrow W^*(n,\mu^*_n)\cap\mathfrak{M}.
\end{align*} 
\end{corollary}

%%%%%%%%%%%%%%%%%%%%%%%%%%%%%%%%%%%%%%%%%%%%%%%%%%%%%%%%%%%%%%%%%%%%%%%%%%%%%%%%%%%%%%%%%%%%%
\subsection{Path-connected components and piecewise complex-analytic paths}\label{subsect-pccapcap}
%%%%%%%%%%%%%%%%%%%%%%%%%%%%%%%%%%%%%%%%%%%%%%%%%%%%%%%%%%%%%%%%%%%%%%%%%%%%%%%%%%%%%%%%%%%%%
Let $\mu\in\mathbb{N}$  and $\mu^*:=(\mu^{(n)},\mu^{(n-1)},\ldots,\mu^{(1)})\in \mathbb{N}^n$ ($n\geq 1$). Again, put $\mu^{\mbox{\tiny max}}:=\mbox{max}\, \{\mu^{(n)},\ldots,\mu^{(1)}\}$.

Piecewise complex-analytic paths in $\mathfrak{M}(\mu)$ are defined as follows.

\begin{definition}\label{def-pcap}
Let $f$ and $f'$ be two polynomial functions vanishing at the origin and lying in the same path-connected component of the $\mu$-constant stratum $\mathfrak{M}(\mu)$ of $\mathfrak{M}$ (as germs of analytic functions at the origin). We say that $f$ and $f'$ can be joined by a \emph{piecewise complex-analytic path} in $\mathfrak{M}(\mu)$ if there exists a continuous path
\begin{equation*}
\gamma\colon [0,1]\to W(n,m,\mu)\cap \mathfrak{M}
\end{equation*}
for some integer $m\geq\mu$ such that:
\begin{enumerate}
\item
$\gamma(0)=f$ and $\gamma(1)=f'$ (in particular, this implies $f,f'\in W(n,m,\mu)$);
\item
there is a partition $0=s_0<s_1<\cdots<s_{q_0}=1$ of $[0,1]$, and for each $0\leq q\leq q_0-1$, there exists an open subset $U_q\subseteq\mathbb{C}$ containing $[s_q,s_{q+1}]$ together with a complex-analytic map 
\begin{equation*}
\tilde\gamma_q\colon U_q\to W(n,m,\mu)\cap \mathfrak{M}
\end{equation*}
 such that $\tilde\gamma_q\vert_{[s_q,s_{q+1}]}=\gamma\vert_{[s_q,s_{q+1}]}$.
\end{enumerate}
\end{definition}

A path $\gamma$ as above is called a \emph{piecewise complex-analytic path} between $f$ and $f'$.
Note that if $f$ and $f'$ can be joined by a piecewise complex-analytic path 
\begin{equation*}
\gamma\colon [0,1]\to W(n,m,\mu)\cap \mathfrak{M},
\end{equation*}
 then, by Proposition \ref{mucs-lemmaequiv}, the Milnor number $\mu_{\mathbf{0}}(\gamma(s))$ is independent of $s\in [0,1]$. This justifies the terminology that $\gamma$ is a path \emph{in the $\mu$-constant statum $\mathfrak{M}(\mu)$}.

Piecewise complex-analytic paths in $\mathfrak{M}(\mu^*)$ are defined similarly, replacing $W(n,m,\mu)$ by $W^*(n,m,\mu^*)$ and changing the inequality $m\geq\mu$ into $m\geq\mu^{\mbox{\tiny max}}$ in Definition \ref{def-pcap}.
In this case, if $f$ and $f'$ can be joined by a piecewise complex-analytic path 
\begin{equation*}
\gamma\colon [0,1]\to W^*(n,m,\mu^*)\cap \mathfrak{M},
\end{equation*}
 then, by Proposition \ref{mucs-lemmaequiv-2}, the $\mu^*$-sequence of $\gamma(s)$ is independent of $s\in [0,1]$.

The next proposition is also an important step in the proof of Theorem \ref{mt3}.

\begin{proposition}\label{radpcap}
If $f$ and $f'$ are in the same path-component of $\mathfrak{M}(\mu)$ and if there exist an integer $m\geq\mu$ and a continuous map 
\begin{equation*}
\varrho\colon [0,1]\to W(n,m,\mu)\cap \mathfrak{M}
\end{equation*}
 with $\varrho(0)=f$ and $\varrho(1)=f'$, then there also exists a continuous map 
\begin{equation*}
\gamma\colon [0,1]\to W(n,m,\mu)\cap \mathfrak{M}
\end{equation*}
 satisfying the conditions (1) and (2) of Definition \ref{def-pcap}.
\end{proposition}

A similar statement also holds true if we replace $\mathfrak{M}(\mu)$ and $W(n,m,\mu)$ by $\mathfrak{M}(\mu^*)$ and $W^*(n,m,\mu^*)$ and if we change the inequality $m\geq\mu$ into $m\geq\mu^{\mbox{\tiny max}}$ both in Proposition \ref{radpcap} and Definition \ref{def-pcap}. The proof is similar to that of Proposition \ref{radpcap}.

\begin{proof}[Proof of Proposition \ref{radpcap}]
Clearly, the assertion is true if the semi-algebraic set 
\begin{equation*}
W(n,m,\mu)\cap \mathfrak{M}
\end{equation*} 
is smooth. If it is singular, then we can reduce the proof to the smooth case by the following argument. First, observe that each point $x$ of the image $\mbox{im}(\varrho)$ has an open neighbourhood $U_x\subseteq P(n,m)\equiv\mathbb{C}^N$ such that the intersection $\mbox{im}(\varrho)\cap U_x$ is contained in an irreducible $k$-dimensional algebraic subvariety $V_x$ of $W(n,m,\mu)\cap \mathfrak{M}$ (for some integer $k$). By the Noether normalization theorem, for each point $y$ of such a variety $V_x$, there is an open neighbourhood $O_y\subseteq \mathbb{C}^N$ and a finite branched covering 
\begin{equation*}
\pi_{x,y}\colon V_x\cap O_y\to U\subseteq\mathbb{C}^k,
\end{equation*} 
where $U$ is an open disc of $\mathbb{C}^k$. Using the compactness of $\mbox{im}(\varrho)$, we choose a sufficiently fine partition $0=s_0<s_1<\cdots<s_{q_0}=1$ of $[0,1]$ so that for each $q$ there exist $x,y$ with 
\begin{equation*}
\varrho([s_q,s_{q+1}])\subseteq V_{x,y}:=V_x\cap O_y.
\end{equation*} 
 Let $\varrho_q$ be the restriction of $\varrho$ to $[s_q,s_{q+1}]$, and 
let $L\subseteq U$ be (the trace on $U$ of) a complex line through $\pi_{x,y}\circ\varrho_q(s_q)$ and $\pi_{x,y}\circ\varrho_q(s_{q+1})$.
The inverse image $\pi_{x,y}^{-1}(L)$ of $L$ by $\pi_{x,y}$ is an algebraic variety of complex dimension $1$, and we easily show that $\varrho_q$ is homotopic to a path contained in $\pi_{x,y}^{-1}(L)$ by a homotopy leaving the ends $\varrho(s_q)$ and $\varrho(s_{q+1})$ fixed. We still denote by $\varrho_q$ the path of $\pi_{x,y}^{-1}(L)$ obtained in this way, and
we consider a normalization 
\begin{equation*}
\tau\colon N(\pi_{x,y}^{-1}(L))\to \pi_{x,y}^{-1}(L).
\end{equation*} 
Then $\varrho_q$ can be lifted to a path $\varsigma_q$ in $N(\pi_{x,y}^{-1}(L))$, and since $N(\pi_{x,y}^{-1}(L))$ is smooth and the problem is solved in this case, we can find an open subset $U_q\subseteq\mathbb{C}$ containing $[s_q,s_{q+1}]$ together with a complex-analytic map $\tilde\varsigma_q\colon U_q\to N(\pi_{x,y}^{-1}(L))$ such that
\begin{equation*}
\tilde\varsigma_q\vert_{\,[s_q,s_{q+1}]}=\varsigma_q.
\end{equation*} 
The desired complex-analytic map $\tilde\gamma_q$ is given by the composite $\tilde\gamma_q:=\tau\circ\tilde\varsigma_q$ while $\gamma$ is the continuous path defined on each $[s_q,s_{q+1}]$ by the restriction $\tilde\gamma_q\vert_{\, [s_q,s_{q+1}]}$.
\end{proof}

We can now state the main result of this section.

\begin{theorem}\label{mt3}
Let $\mu\in\mathbb{N}$  and $\mu^*:=(\mu^{(n)},\mu^{(n-1)},\ldots,\mu^{(1)})\in \mathbb{N}^n$ ($n\geq 1$), and let $f$ and $f'$ be polynomial functions on $\mathbb{C}^n$ 
such that $f(\mathbf{0})=f'(\mathbf{0})=0$. 
\begin{enumerate}
\item
If $f$ and $f'$ (as germs in~$\mathfrak{M}$) are in the same path-connected component of the $\mu$-constant stratum $\mathfrak{M}(\mu)$, then they can be joined by a piecewise complex-analytic path 
\begin{equation*}
\gamma\colon [0,1]\to W(n,m,\mu)\cap\mathfrak{M}
\end{equation*}
for any integer $m\geq\mbox{\emph{max}}\{\deg f, \deg f',\mu\}$.
\item
Similarly, if $f$ and $f'$ are in the same path-connected component of the $\mu^*$-constant stratum $\mathfrak{M}(\mu^*)$, then they can be joined by a piecewise complex-analytic path 
\begin{equation*}
\gamma\colon [0,1]\to W^*(n,m,\mu^*)\cap\mathfrak{M}
\end{equation*}
for any integer $m\geq\mbox{\emph{max}}\{\deg f, \deg f',\mu^{(n)},\ldots,\mu^{(1)}\}$.
\end{enumerate}
\end{theorem} 

\begin{proof}
To show the first item, let $f$ and $f'$ be polynomial functions lying in the same path-connected component of the $\mu$-constant stratum of $\mathfrak{M}$. Then there is a continuous path $\varrho\colon [0,1]\to \mathfrak{M}$, $s\mapsto\varrho(s)$, such that $\varrho(0)=f$, $\varrho(1)=f'$ and $\mu_{\mathbf{0}}(\varrho(s))=\mu_{\mathbf{0}}(f)=\mu_{\mathbf{0}}(f')=:\mu$. In other words, $\varrho(s)\in W(n,\mu)\cap \mathfrak{M}=\mathfrak{M}(\mu)$. Take any $m\geq\mbox{max}\{\deg f,\deg f',\mu\}$. Then Proposition \ref{mucs-lemmaequiv} shows that $\pi_m(\varrho(s))\in W(n,m,\mu)\cap \mathfrak{M}$ for any $s\in [0,1]$, and since 
\begin{equation*}
\pi_m(\varrho(0))=\pi_m(f)=f
\quad\mbox{and}\quad
\pi_m(\varrho(1))=\pi_m(f')=f',
\end{equation*} 
we have that $s\mapsto\pi_m(\varrho(s))$ is a path in $W(n,m,\mu)\cap \mathfrak{M}$ from $f$ to $f'$. Thus, by Proposition \ref{radpcap}, there is also a path 
\begin{equation*}
\gamma\colon [0,1]\to W(n,m,\mu)\cap\mathfrak{M}
\end{equation*}
 satisfying the conditions (1) and (2) of Definition \ref{def-pcap} (i.e., $f$ and $f'$ can be joined by a piecewise complex-analytic path in $\mathfrak{M}(\mu)$).

To prove the second item, let $f$ and $f'$ be polynomial functions lying in the same path-connected component of the $\mu^*$-constant stratum of $\mathfrak{M}$. Then there is a continuous path $\varrho\colon [0,1]\to \mathfrak{M}$, $s\mapsto\varrho(s)$, such that $\varrho(0)=f$, $\varrho(1)=f'$ and the $\mu^*$-sequence of $\varrho(s)$ is given by $\mu^*\equiv\mu^*_n:=(\mu^{(n)},\ldots,\mu^{(1)})$ for any $s\in [0,1]$, where $\mu^{(i)}:=\mu_{\mathbf{0}}^{(i)}(f)=\mu_{\mathbf{0}}^{(i)}(f')$.
In other words, $\varrho(s)\in W^*(n,\mu^*)\cap\mathfrak{M}=\mathfrak{M}(\mu^*)$.  Take any $m\geq\mbox{max}\{\deg f,\deg f',\mu^{(n)},\ldots,\mu^{(1)}\}$. Then Proposition \ref{mucs-lemmaequiv-2} shows that $\pi_m(\varrho(s))\in W^*(n,m,\mu^*)\cap \mathfrak{M}$ for any $s\in [0,1]$, and since 
\begin{equation*}
\pi_m(\varrho(0))=\pi_m(f)=f
\quad\mbox{and}\quad
\pi_m(\varrho(1))=\pi_m(f')=f',
\end{equation*}
we have that $s\mapsto\pi_m(\varrho(s))$ is a path in $W^*(n,m,\mu^*)\cap \mathfrak{M}$ from $f$ to $f'$.
Thus, by the $\mu^*$ version of Proposition \ref{radpcap} (see the comment after it), there is also a path 
\begin{equation*}
\gamma\colon [0,1]\to W^*(n,m,\mu^*)\cap \mathfrak{M}
\end{equation*}
 satisfying the conditions (1) and (2) of Definition \ref{def-pcap} with $\mathfrak{M}(\mu^*)$ instead of $\mathfrak{M}(\mu)$ and $W^*(n,m,\mu^*)$ instead of $W(n,m,\mu)$ (i.e., $f$ and $f'$ can be joined by a piecewise complex-analytic path in $\mathfrak{M}(\mu^*)$).
\end{proof}

%%%%%%%%%%%%%%%%%%%%%%%%%%%%%%%%%%%%%%%%%%%%%%%%%%%%%%%%%%%%%%%%%%%%%%%%%%%%%%%%%%%%%%%%%%%%%
\section{Construction of $\mu^*$-Zariski pairs of surfaces}\label{sect-mszph}
%%%%%%%%%%%%%%%%%%%%%%%%%%%%%%%%%%%%%%%%%%%%%%%%%%%%%%%%%%%%%%%%%%%%%%%%%%%%%%%%%%%%%%%%%%%%%
In this last section, we construct examples of $\mu^*$-Zariski pairs of surfaces. The main tools we use are Theorems \ref{mt1} and \ref{mt3} and the Oka formula \eqref{okaformula11}--\eqref{okaformula13} for the zeta-function.

\subsection{Zeta-multiplicity and zeta-multiplicity factor}
Let $h(\mathbf{z})$ be a non-constant analytic function defined in a neighbourhood of the origin of $\mathbb{C}^n$ and such that $h(\mathbf{0})=0$. By the A'Campo--Oka formula \eqref{ACampoformula}, the zeta-function $\zeta_{h,\mathbf{0}}(t)$ of the monodromy associated with the Milnor fibration of $h$ at $\mathbf{0}$ can be uniquely written as
\begin{equation}\label{sectmszp-expzf}
\zeta_{h,\mathbf{0}}(t)=\prod_{i=1}^{\ell} (1-t^{d_i})^{\nu_i},
\end{equation}
where $d_1,\ldots,d_{\ell}$ are mutually disjoint and $\nu_1,\ldots,\nu_{\ell}$ are non-zero integers. Then, as in \cite{O3}, we define the \emph{zeta-multiplicity} associated with the function $h$ as the integer 
\begin{equation*}
m_\zeta(h):=\mbox{min} \{d_i\, ;\, 1\leq i\leq \ell\}.
\end{equation*}
Observe that $m_\zeta(h)\geq \mbox{mult}_{\mathbf{0}}(h)$, where $\mbox{mult}_{\mathbf{0}}(h)$ is the usual multiplicity of $h$ at $\mathbf{0}$. The factor $(1-t^{d_i})^{\nu_i}$ in \eqref{sectmszp-expzf} that corresponds to the integer $i$ for which $d_i=m_\zeta(h)$ is called the \emph{zeta-multiplicity factor} of $\zeta_{h,\mathbf{0}}(t)$.  

\subsection{Examples of $\mu^*$-Zariski pairs of surfaces}
Now, assume that the number $n$ of complex variables is $3$, and consider two reduced, convenient, homogeneous, polynomial functions $f_0(z_1,z_2,z_3)$ and $f_1(z_1,z_2,z_3)$
of degree $d$ such that the corresponding curves $C_0$ and $C_1$ in the complex projective plane~$\mathbb{P}^2$ makes a ``Zariski pair'' \textemdash\ that is, there is a homeomorphism between the pairs $(N(C_0),C_0)$ and $(N(C_1),C_1)$ for some regular neighbourhoods $N(C_0)$ and $N(C_1)$ of $C_0$ and $C_1$, respectively, but there is no homeomorphism between the pairs $(\mathbb{P}^2,C_0)$ and $(\mathbb{P}^2,C_1)$. We assume that the singularities of the curves are Newton non-degenerate in some suitable local coordinates (in particular this is always the case if we are dealing with ``simple'' singularities in the sense of Arnol'd \cite{Arnold}). We also assume that these singularities are located in $z_1z_2z_3\not=0$. 
In particular, this implies that the functions $f_0$ and $f_1$ are weakly almost Newton non-degenerate (see Definition \ref{defowanndf}), and by an argument similar to that given in Example~\ref{sect-prel-example}, we see that they are also Newton pre-non-degenerate (see Definition \ref{def-psND}). 
Still by Example \ref{sect-prel-example}, we have that for any integer $m\geq 1$, the polynomial functions
\begin{equation}\label{civa}
g_0(z_1,z_2,z_3):=f_0(z_1,z_2,z_3)+z_1^{d+m}
\quad\mbox{and}\quad
g_1(z_1,z_2,z_3):=f_1(z_1,z_2,z_3)+z_1^{d+m}
\end{equation}
are almost Newton non-degenerate, and by Theorem \ref{mt1}, we know that these functions have an isolated singularity at the origin. The proof of Theorem \ref{mt1} also shows that 
\begin{equation*}
\zeta_{g_0,\mathbf{0}}(t)=\zeta_{g_1,\mathbf{0}}(t).
\end{equation*}
 As in \cite{O2,O3}, we call such a pair $(g_0,g_1)$ a \emph{Zariski pair of surfaces} (or a \emph{Zariski pair of links}).
The main result of this section is the following theorem.

\begin{theorem}\label{mt2}
Under the above assumptions, the Zariski pair of surfaces $(g_0,g_1)$ is in fact a $\mu^*$-Zariski pair of surfaces, that is, the functions $g_0(\mathbf{z})$ and $g_1(\mathbf{z})$ have the same Teissier's $\mu^*$-sequence but they do not belong to the same path-connected component of the $\mu^*$-constant stratum of $\mathfrak{M}$.
\end{theorem}

\begin{remark}
Of course, as above, a similar result still holds true if we replace the term $z_1^{d+m}$ in \eqref{civa} either by $z_2^{d+m}$ or by $z_3^{d+m}$.
\end{remark}

\begin{proof}[Proof of Theorem \ref{mt2}]
First, we show that  $g_0$ and $g_1$ have the same $\mu^*$-sequence at $\mathbf{0}$.
By Theorem \ref{mt1}, the Milnor numbers $\mu_\mathbf{0}(g_0)$ and $\mu_\mathbf{0}(g_1)$ of $g_0$ and $g_1$ at $\mathbf{0}$ are given by the formula \eqref{edmt1}. Since $C_0$ and $C_1$ have regular neighbourhoods $N(C_0)$ and $N(C_1)$ such that $(N(C_0),C_0)\simeq(N(C_1),C_1)$, they have the same ``combinatoric'' (see \cite{Artal,Artaletal} for the definition), so that the local Milnor numbers $\mu_{\mathbf{p}}$ that appear in the formula \eqref{edmt1} are the same for both $C_0$ and $C_1$. It follows that 
\begin{equation*}
\mu_\mathbf{0}(g_0)=\mu_\mathbf{0}(g_1)=(d-1)^3+m \mu^{\mbox{\tiny tot}},
\end{equation*}
where $\mu^{\mbox{\tiny tot}}$ is the (finite) sum of the local Milnor numbers at the singular points of $C_0$ (or equivalently, at the singular points of $C_1$), that is, the sum of the $\mu_{\mathbf{p}}$'s.

Now, for a generic hyperplane $H$ through the origin $\mathbf{0}\in\mathbb{C}^3$, the restriction $f_l\vert_{H}$, $l\in\{0,1\}$, is a homogeneous polynomial of degree $d$ with an isolated singularity at the origin, so that its Milnor number at $\mathbf{0}$ is $\mu_{\mathbf{0}}(f_l\vert_{H})=(d-1)^2$. Clearly, $f_l\vert_{H}$ is Newton non-degenerate, and since the term $z_1^{d+m}$ is above the Newton boundary $\Gamma(g_l\vert_{H})=\Gamma(f_l\vert_{H})$, the function $g_l\vert_{H}$ is Newton non-degenerate too. Thus, its Milnor number at $\mathbf{0}$ is determined by $\Gamma(g_l\vert_{H})$, and hence we have 
\begin{equation*}
\mu_{\mathbf{0}}^{(2)}(g_l):=\mu_{\mathbf{0}}(g_l\vert_{H})=\mu_{\mathbf{0}}(f_l\vert_{H})=(d-1)^2 \quad\mbox{for}\quad l\in\{0,1\}.
\end{equation*}

Finally, since the multiplicities of $g_0$ and $g_1$ at $\mathbf{0}$ are equal to $d$, it follows that $g_0$ and $g_1$ have the same $\mu^*$-sequence at $\mathbf{0}$, namely for any $l\in\{0,1\}$ we have
\begin{equation*}
\mu_{\mathbf{0}}^*(g_l):=(\mu_{\mathbf{0}}(g_l),\mu_{\mathbf{0}}^{(2)}(g_l),\mbox{mult}_{\mathbf{0}}(g_l)-1)=((d-1)^3+m \mu^{\mbox{\tiny tot}},(d-1)^2,d-1).
\end{equation*}

Now, to prove that $g_0$ and $g_1$ lie in different path-connected components of the $\mu^*$-constant stratum $\mathfrak{M}(\mu^*)$ of $\mathfrak{M}$, we argue by contradiction. Suppose they belong to the same path-connected component. Then, by Theorem \ref{mt3}, there exists a piecewise complex-analytic path
\begin{equation*}
\gamma\colon [0,1]\to W^*(3,m',\mu^*)\cap\mathfrak{M}
\end{equation*} 
connecting $g_0$ and $g_1$, where $\mu^*$ denotes the triple $(\mu_{\mathbf{0}}^{(3)}(g_l),\mu_{\mathbf{0}}^{(2)}(g_l),\mu_{\mathbf{0}}^{(1)}(g_l))$ and where $m'$ is an integer\footnote{We use the letter $m'$ in $W^*(n,m',\mu^*)=W^*(3,m',\mu^*)$, the letter $m$ being already used in the present section with a different meaning.} satisfying
\begin{align*}
m'\geq\mbox{max}\{\deg g_0,\deg g_1,\mu_{\mathbf{0}}^{(3)}(g_l),\mu_{\mathbf{0}}^{(2)}(g_l),\mu_{\mathbf{0}}^{(1)}(g_l)\}.
\end{align*} 
In other words, there is a piecewise complex-analytic family $\{g_s\}_{0\leq s\leq 1}$ of functions $g_s:=\gamma(s)$ connecting $g_0$ and $g_1$ and such that the $\mu^*$-sequence of $g_s$ is independent of $s\in [0,1]$.

As a part of the $\mu^*$-constancy, the multiplicity $\mbox{mult}_{\mathbf{0}}(g_s)$ of $g_s$ at $\mathbf{0}$ is independent of $s\in [0,1]$, and hence for each $s$, the initial polynomial $\mbox{in}(g_s)$ of $g_s$ has degree $d$. Moreover this polynomial satisfies the following property.
 
\begin{claim}\label{claim53}
For each $s\in [0,1]$, the homogeneous polynomial $\mbox{\emph{in}}(g_s)$ is reduced, so that the projective curve $C_{s}\subseteq \mathbb{P}^2$ defined by $\mbox{\emph{in}}(g_s)$ has only isolated singularities.
\end{claim}

\begin{proof}
We argue by contradiction. Suppose there exists $s_0\in [0,1]$ such that $\mbox{in}(g_{s_0})$ is not reduced (i.e., $C_{s_0}$ has non-isolated singularities). Then, for a generic linear plane $H$ of $\mathbb{C}^3$, there are coordinates $(x,y)$ for $H$ and linear forms $\ell_1(x,y),\ldots,\ell_q(x,y)$ such that 
\begin{equation*}
\mbox{in}(g_{s_0})\vert_H(x,y)=\ell_1(x,y)^{p_1}\cdots\ell_q(x,y)^{p_q}
\end{equation*}
with $p_1\geq \cdots\geq p_q$ and $p_1\geq 2$. By a linear change of coordinates, we may assume that $\ell_1(x,y)\equiv x$, so that 
\begin{equation*}
\mbox{in}(g_{s_0})\vert_H(x,y)=x^{p_1}h(x,y),
\end{equation*}
where $h$ is a homogeneous polynomial of degree $d-p_1$ (in particular, $\mbox{in}(g_{s_0})\vert_H$ is not convenient with respect to the coordinates $(x,y)$).
By adding monomials of the form $x^{\alpha}$ and $y^{\beta}$ for $\alpha,\, \beta$ large enough, we may also assume that $g_{s_0}\vert_H$ is convenient.
Now since the integral point $(1,d-1)$ is not on the Newton boundary $\Gamma(\mbox{in}(g_{s_0})\vert_H)$ of $\mbox{in}(g_{s_0})\vert_H$ with respect to the coordinates $(x,y)$, it follows\footnote{Let us briefly show it, for instance, in the special case where the Newton boundaries are as in Figure \ref{figure2}, the general case being completely similar. Clearly, in this case,
\begin{equation*}
\nu(\Gamma_{\!-}(g_{s_0}\vert_H))=2S'-(d+c)-(d+e)+1,
\end{equation*}
where $S'=S+c\, q/2+e\, p/2$ with $p\geq p_1\geq 2$ and $S$ is the area of the triangle $(0,d,d)$. Similarly, $\nu(\Gamma_{\!-}(g_{0}\vert_H))=2S-2d+1$. Since $p\geq 2$, it follows that
\begin{equation*}
\nu(\Gamma_{\!-}(g_{s_0}\vert_H))-\nu(\Gamma_{\!-}(g_{0}\vert_H))=c(q-1)+e(p-1)>0
\end{equation*}
 (note that if $q=0$, then $c=0$, and the above inequality still holds true).} that 
\begin{equation*}
\nu(\Gamma_{\!-}(g_{s_0}\vert_H))>\nu(\Gamma_{\!-}(g_{0}\vert_H))
\end{equation*}
(see Figure \ref{figure2}).
\begin{figure}[t]
\includegraphics[scale=2]{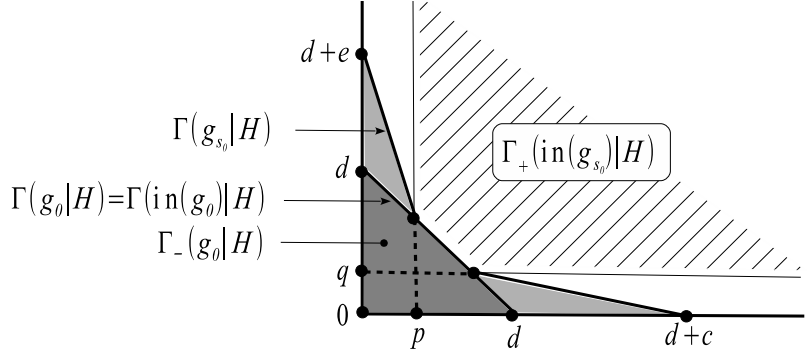}
\caption{Newton boundaries and cones over them}
\label{figure2}
\end{figure}
Here, $\nu(\cdot)$ denotes the Newton number (see \cite{K} for the definition) and $\Gamma_{\!-}(g_{s_0}\vert_H)$ stands for the cone over $\Gamma(g_{s_0}\vert_H)$ with the origin as vertex. (Again, $\Gamma(g_{s_0}\vert_H)$ denotes the Newton boundary of $g_{s_0}\vert_H$ with respect to the coordinates $(x,y)$.) The polyhedron $\Gamma_{\!-}(g_{0}\vert_H)$ is defined similarly.
Since $\mu_{\mathbf{0}}(g_{s_0}\vert_H)\geq \nu(\Gamma_{\!-}(g_{s_0}\vert_H))$ (see \cite[Th\'eor\`eme~1.10]{K}), altogether we have
\begin{equation*}
\mu^{(2)}_{\mathbf{0}}(g_{s_0})=\mu_{\mathbf{0}}(g_{s_0}\vert_H)\geq \nu(\Gamma_{\!-}(g_{s_0}\vert_H))>\nu(\Gamma_{\!-}(g_{0}\vert_H))=(d-1)^2=\mu^{(2)}_{\mathbf{0}}(g_{0}),
\end{equation*}
which is a contradiction to the $\mu^*$-constancy.
\end{proof}

\begin{claim}\label{zmcc}
The zeta-function $\zeta_{g_s,\mathbf{0}}(t)$ of the monodromy associated with the Milnor fibration of $g_s$ at $\mathbf{0}$ is independent of $s\in [0,1]$. In particular, the zeta-multiplicity $m_{\zeta}(g_s)$ associated with the function $g_s$ and the zeta-multiplicity factor of $\zeta_{g_s,\mathbf{0}}(t)$ are both independent of $s\in [0,1]$.
\end{claim}

\begin{proof}
This claim follows, for instance, from the following result of Teissier \cite{Teissier}: if $h$ and $h'$ are two analytic functions such that $h(\mathbf{0})=h'(\mathbf{0})=0$ and $h$ and $h'$ can be connected by a $\mu^*$-constant piecewise complex-analytic path, then for any sufficiently small $\varepsilon>0$, the pairs 
\begin{equation*}
(S^{2n-1}_\varepsilon,K_{h})
\quad\mbox{and}\quad
(S^{2n-1}_\varepsilon,K_{h'})
\end{equation*}
are diffeomorphic; here, $S^{2n-1}_\varepsilon$ stands for the sphere in $\mathbb{C}^n$ with centre $\mathbf{0}$ and radius $\varepsilon$, and $K_{h}$ denotes the link of $V(h):=h^{-1}(0)$, that is, $K_{h}:=S_\varepsilon^{2n-1}\cap V(h)$ for $\varepsilon>0$ small enough (of course, $K_{h'}$ is defined similarly). 

Alternatively, Claim \ref{zmcc} can also be deduced from \cite[Lemma~12]{O3} which asserts that the independence with respect to $s$ of the Milnor number $\mu_{\mathbf{0}}(g_s)$ implies the independence of the zeta-function $\zeta_{g_s,\mathbf{0}}(t)$. This latter lemma follows from the L\^e--Ramanujam theorem \cite{LR} together with the Sebastiani--Thom join theorem \cite{ST} and its generalization by Sakamoto \cite{S} (for details, see \cite{O3}).
\end{proof}

Now, the expression for the zeta-function $\zeta_{g_0,\mathbf{0}}(t)$ given by the formulas \eqref{okaformula12} and \eqref{okaformula13} shows that $m_{\zeta}(g_0)\leq d$, but as observed above, we also have $m_{\zeta}(g_0)\geq \mbox{mult}_{\mathbf{0}}(g_0)=d$, so that altogether we get $m_{\zeta}(g_0)=d$. Thus, by Claim \ref{zmcc}, for any  $s\in [0,1]$ we have $m_{\zeta}(g_s)=d$. 
Moreover the zeta-multiplicity factor of $\zeta_{g_s,\mathbf{0}}(t)$ satisfies the following property.

\begin{claim}\label{zmcc2}
For any $s\in [0,1]$, the zeta-multiplicity factor of $\zeta_{g_s,\mathbf{0}}(t)$ is given by
\begin{equation*}
(1-t^d)^{c+\mu^{\mbox{\tiny \emph{tot}}}(C_{s})}
\end{equation*}
 where $c$ is a constant independent of $s\in [0,1]$ and $\mu^{\mbox{\tiny \emph{tot}}}(C_{s})$ is the total Milnor number of $C_s$ (i.e., the sum of all local Milnor numbers associated with the singularities of $C_s$). In particular, since the zeta-multiplicity factor of $\zeta_{g_s,\mathbf{0}}(t)$ is also independent of $s$ (see Claim~\ref{zmcc}), we have that $\mu^{\mbox{\tiny \emph{tot}}}(C_{s})$ is independent of $s$ as well.
\end{claim} 

\begin{proof}
By a linear change of coordinates, we may assume that the initial polynomial $\mbox{in}(g_s)$ is convenient and Newton non-degenerate on each coordinate subspace of dimension $2$ (i.e., $\mbox{in}(g_s)\vert_{\mathbb{C}^I}$ is Newton non-degenerate for any $I\subseteq \{1,2,3\}$ with $|I|=2$). We may also assume that the singular points of $\mbox{in}(g_s)$ are not located on the coordinate axes. Consider the regular simplicial cone subdivision $\Sigma^*$ of $W^+_{\mathbb{R}}$ whose vertices are the canonical weight vectors $\mathbf{e}_1,\mathbf{e}_2,\mathbf{e}_3$ together with the weight vector $\mathbf{w}:={}^t(1,1,1)$, and look at the associated toric modification $\hat{\pi}\colon X\to \mathbb{C}^3$, which is nothing but the ordinary point blowing-up with centre at the origin. The multiplicity of $\hat g_s:=\hat\pi^*g_s$ along the compact exceptional divisor $\hat E(\mathbf{w})\simeq\mathbb{P}^2$ is the degree $d$ of $\mbox{in}(g_s)$. If the face function $(g_s)_{\Delta(\mathbf{w};g_s)}$ is not Newton non-degenerate, then the strict transform $\widetilde{V}$ of $V(g_s)$ by $\hat\pi$ may have singularities in $\hat E(\mathbf{w})$. By Claim \ref{claim53}, $E(\mathbf{w})=\widetilde{V}\cap\hat E(\mathbf{w})$ (which is given by $\mbox{in}(g_s)$) has only a finite number of isolated singularities $a_1(s),\ldots,a_{j_s}(s)$. To get a good resolution of $g_s$, we need further blowing-ups over these singular points. For each $1\leq j\leq j_s$, let $\hat\omega_j\colon Y_j\to X$ be the resolution with centre $a_j(s)$ which resolves $\hat g_s$ at $a_j(s)$. Denote by $\hat\omega\colon Y\to X$ the canonical gluing of the union (over all $1\leq j\leq j_s$) of these resolutions. 
Then
\begin{equation*}
\hat\Pi:=\hat\pi\circ\hat\omega\colon Y \xrightarrow{\ \hat\omega \ } X \xrightarrow{\ \hat\pi \ } \mathbb{C}^3
\end{equation*} 
gives a good resolution of $g_s$. The exceptional divisors of $\hat\Pi$ are the exceptional divisors $D_1,\ldots,D_r$ of $\hat\omega$ and the pull back $\hat E_Y(\mathbf{w})$ of $\hat E(\mathbf{w})$ by $\hat\omega$. Now we observe that if $m_i$ denotes the multiplicity of $\hat\Pi^*g_s$ along $D_i$ ($1\leq i\leq r$), then, by Remark \ref{remark-multetmi}, $m_i$ is greater than the multiplicity $d$ of $\hat\pi^* g_s$ along $\hat E(\mathbf{w})$.  We also notice that since $\hat E (\mathbf{w})\setminus \widetilde{V}$ is non-singular and does not contain any centre $a_j(s)$ ($1\leq j \leq j_s$), there is a canonical diffeomorphism 
\begin{align*}
\hat\omega\colon \hat E_Y (\mathbf{w})\Bigg\backslash \Bigg(\widetilde{V}_{Y}\cup \bigcup_{1\leq j\leq r} D_j\Bigg)\overset{\sim}{\longrightarrow}\hat E (\mathbf{w})\setminus \widetilde{V},
\end{align*}
where $\widetilde{V}_{Y}$ denotes the strict transform of $\widetilde{V}$ by $\hat\omega$. 
Altogether, this implies that the zeta-multiplicity factor of $\zeta_{g_s,\mathbf{0}}(t)$ is given by
\begin{equation}\label{aoa2}
(1-t^d)^{-\chi(\hat E (\mathbf{w})\setminus \widetilde{V})}.
\end{equation}
To show that $-\chi(\hat E (\mathbf{w})\setminus \widetilde{V})=c+\mu^{\mbox{\tiny tot}}(C_{s})$, where $c$ is a constant independent of $s$, we consider an analytic deformation $\{(g_s)_u\}_{|u|<1}$ of $g_s$ obtained from a small perturbation of the coefficients of the face function $(g_s)_{\Delta(\mathbf{w};g_s)}$ such that $(g_s)_u$ is Newton non-degenerate for all $u\not=0$ (as in \S \ref{subsect-OF}). Then, for such a non-zero $u$, we have
\begin{equation*}
\chi(\hat E(\mathbf{w})\cap\widetilde{V})=
\chi(\hat E(\mathbf{w})\cap\widetilde{V((g_s)_u)})+\sum_{\mathbf{p}\in \mbox{\tiny Sing}(E(\mathbf{w}))} \mu_{\mathbf{p}}=\chi(\hat E(\mathbf{w})\cap\widetilde{V((g_s)_u)})+\mu^{\mbox{\tiny tot}}(C_s),
\end{equation*}
and hence,
\begin{equation*}
\chi(\hat E(\mathbf{w})\setminus\widetilde{V})=
\chi(\hat E(\mathbf{w})\setminus\widetilde{V((g_s)_u)})-\mu^{\mbox{\tiny tot}}(C_s),
\end{equation*}
where $\widetilde{V((g_s)_u)}$ is the strict transform of $V((g_s)_u):=(g_s)_u^{-1}(0)$.
Of course, 
\begin{equation*}
\chi(\hat E(\mathbf{w})\setminus\widetilde{V((g_s)_u)})
\end{equation*}
 is independent of $u\not=0$, but the key observation is that $\chi(\hat E(\mathbf{w})\setminus\widetilde{V((g_s)_u)})$ is also independent of $s\in [0,1]$ (remind that $\mbox{in}(g_s)$ is convenient). This shows that the zeta-multiplicity factor of $\zeta_{g_s,\mathbf{0}}(t)$ is written as
\begin{equation*}
(1-t^d)^{c+\mu^{\mbox{\tiny tot}}(C_{s})},
\end{equation*}
where $c:=-\chi(\hat E(\mathbf{w})\setminus\widetilde{V((g_s)_u)})$ is independent of $s\in [0,1]$ as desired.
\end{proof}

\begin{remark}
Actually, the integer $c$ of Claim \ref{zmcc2} is given by $c=d^2+3d-3$. Indeed, by Claim \ref{zmcc}, we know that the zeta-multiplicity factor of $\zeta_{g_s,\mathbf{0}}(t)$ is independent of $s$, so in order to prove the equality $c=d^2+3d-3$, it suffices to show that the zeta-multiplicity factor of $\zeta_{g_0,\mathbf{0}}(t)$ is given by 
\begin{equation*}%\label{zmffg0}
(1-t^d)^{d^2+3d-3+\mu^{\mbox{\tiny tot}}(C_{0})}.
\end{equation*}
 By \eqref{okaformula11}--\eqref{okaformula13} and Remark \ref{remark-multetmi}, we easily see that the zeta-multiplicity factor of $\zeta_{g_0,\mathbf{0}}(t)$ is written as
\begin{equation*}
\zeta_{(g_0)_u}(t)\cdot (1-t^d)^{\mu^{\mbox{\tiny tot}}(C_0)}, 
\end{equation*}
where $\{(g_0)_u\}_{|u|<1}$ is a deformation of $g_0$ as above.
But in our case the zeta-function $\zeta_{(g_0)_u}(t)$ is nothing but $\zeta_{f_0,\mathbf{0}}(t)$, which can be calculated using the Varchenko formula \eqref{varchenkoformula} as
\begin{equation*}
\zeta_{(g_0)_u}(t)=\zeta_{f_0,\mathbf{0}}(t)=(1-t^d)^{d^2+3d-3}.
\end{equation*}
\end{remark}

We can now complete the proof of Theorem \ref{mt2}.
For each $s\in [0,1]$, let again $a_1(s),\ldots,a_{j_s}(s)$ denote the singular points of $C_{s}$ and $\mu^{\mbox{\tiny tot}}(C_s)$ denote the total Milnor number of $C_s$, that is, the sum of the local Milnor numbers $\mu(C_s,a_j(s)):=\mu_{a_j(s)}(\mbox{in}(g_s))$ of the singularities $(C_s,a_{j}(s))$ for $1\leq j\leq j_s$.
Observe that if there is a bifurcation of the singularities at $s=s_0$ (for some $s_0\in [0,1]$) in a small ball $B_j$ centred at a singular point $a_j(s_0)$ of $C_{s_0}$ \textemdash\ that is, $a_j(s_0)$ is the only singular point of $C_{s_0}$ in $B_j$ and it is either a ``newly born'' singularity or a singularity obtained as a ``merging'' of several singularities of $C_s$ for $s$ near, but not equal to, $s_0$ (see Figure \ref{figure3}) \textemdash\ then, by \cite[Th\'eor\`eme B]{Le} (see also \cite{B,Lazzeri}), we have
\begin{equation*}\label{plemma-esmn}
\sum_{k=1}^{k_{j}(s)} \mu(C_s,a_{j,k}(s)) < \mu(C_{s_0},a_j(s_0)),
\end{equation*}
where $a_{j,1}(s),\ldots,a_{j,k_{j}(s)}(s)$ are the singular points of the curve $C_s$ in the ball ${B}_j$.
Let
\begin{equation*}
Z:=\{(a,s)\in\mathbb{P}^2\times [0,1]\mid a \mbox{ is a singular point of } C_s\},
\end{equation*}
and let $\mbox{pr}_1\colon Z\to \mathbb{P}^2$ and $\mbox{pr}_2\colon Z\to [0,1]$ be the projections on the first and second factor respectively. Note that $\mbox{pr}_1(Z)=\bigcup_{s\in [0,1]}\{a_1(s),\ldots,a_{j_s}(s)\}$. By the above observation, if there is $s_0\in [0,1]$ such that $C_{s_0}$ gets either newly born singularities or several singularities of $C_s$ (for $s\not= s_0$ near $s_0$) merge into one (i.e., $s_0$ is a point where $\mbox{pr}_2$ fails to be a covering, see Figure \ref{figure3}), then we have $\mu^{\mbox{\tiny tot}}(C_{s})<\mu^{\mbox{\tiny tot}}(C_{s_0})$. However this contradicts Claim \ref{zmcc2}. Thus there is no such an $s_0$, and hence, by \cite{Le2}, the topological type of the pair $(\mathbb{P}^2,C_s)$ is independent of $s\in [0,1]$, so that $(C_0,C_1)$ is not a Zariski pair \textemdash\ a contradiction. 
\begin{figure}[t]
\includegraphics[scale=2]{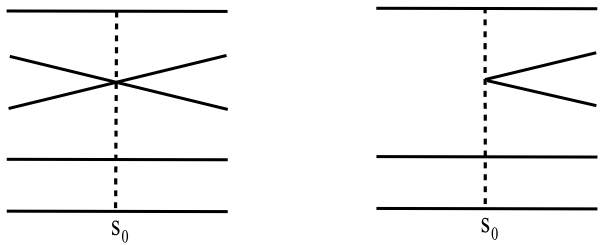}
\caption{Bifurcation of singularities}
\label{figure3}
\end{figure}

This completes the proof of Theorem \ref{mt2}.
\end{proof}

\begin{remark}
Theorem \ref{mt2} remains valid if we assume that the pair of projective curves $(C_0,C_1)$ is only a \emph{weak} Zariski pair instead of a Zariski pair. We recall that $(C_0,C_1)$ is said to be a \emph{weak Zariski pair} if there is a bijection $\phi\colon \mbox{Sing}(C_0)\to\mbox{Sing}(C_1)$ between the singular loci $\mbox{Sing}(C_0)$ and $\mbox{Sing}(C_1)$ of $C_0$ and $C_1$, respectively, such that for any $p\in \mbox{Sing}(C_0)$ the singularities $(C_0,p)$ and $(C_1,\phi(p))$ have the same embedded topological type (i.e., there are neighbourhoods $U_p$ and $U_{\phi(p)}$ of $p$ and $\phi(p)$, respectively, together with a homeomorphism of triples $\phi_p\colon (U_p,C_0\cap U_p,p)\to (U_{\phi(p)},C_1\cap U_{\phi(p)},\phi(p))$) while for any regular neighbourhoods $N(C_0)$ and $N(C_1)$ of $C_0$ and $C_1$, respectively, the pairs $(N(C_0),C_0)$ and $(N(C_1),C_1)$ are not homeomorphic (in particular,
the pairs $(\mathbb{P}^2,C_0)$ and $(\mathbb{P}^2,C_1)$ are not homeomorphic either). 
\end{remark}

\bibliographystyle{amsplain}

\end{document}